\pdfoutput=1

\documentclass[11pt]{article}
\usepackage{a4}
\usepackage{amsmath,amsfonts,amssymb,amscd,amsthm}
\usepackage{latexsym,mathrsfs}
\usepackage[pdftex]{graphicx}
\usepackage{subfigure}
\usepackage{pict2e}
\usepackage{color}
\usepackage{hyperref}
\usepackage{epstopdf}

\newtheorem{theorem}{Theorem}
\newtheorem{proposition}{Proposition}
\newtheorem{definition}{Definition}
\newtheorem{lemma}{Lemma}

\newtheorem{example}{Example}

\newtheorem{alg}{Algorithm}
\newtheorem{cor}{Corollary}
\newtheorem{proper}{Properties}

\newcommand{\Keywords}[1]{\par\noindent
{\small{ \textbf{Key words}\/}: #1}}

\hypersetup{pdftitle={Molodensky},%
    pdfauthor={Heiko Gimperlein},%
    pdfsubject={Molodensky}}

\begin{document}
\title{A Nash--H\"{o}rmander iteration and boundary elements for the Molodensky problem}
\author{Adrian Costea\thanks{Institute of Applied Mathematics, Leibniz University Hannover, Welfengarten 1, 30167 Hannover, Germany, email: \{costea, stephan\}@ifam.uni-hannover.de } \and Heiko Gimperlein\thanks{Department of Mathematical Sciences, University of Copenhagen, Universitetsparken 5, 2100 Copenhagen \O , Denmark, email: gimperlein@math.ku.dk} \and Ernst P.~Stephan$^\ast$}
\maketitle
\begin{abstract}
We investigate the numerical approximation of the nonlinear Molodensky problem, which reconstructs the surface of the earth from the gravitational potential and the gravity vector. The method, based
on a smoothed Nash--H\"{o}rmander iteration, solves a sequence of exterior oblique Robin problems and uses a regularization based on a higher-order heat equation to overcome the loss of derivatives in the surface update.
In particular, we obtain a quantitative a priori estimate for the error after $m$ steps,
justify the use of smoothing operators based on the heat equation, and comment on the accurate evaluation of the Hessian of the gravitational potential on the surface, using a representation in terms of a hypersingular integral. A boundary element method is used to solve the exterior problem.
Numerical results compare the error between the approximation and the exact solution in a model problem.
\end{abstract}
\section{Introduction}

The determination of the shape of the earth and its exterior gravitational field from terrestrial measurements is a basic problem in physical geodesy \cite{c3, Heiskanen1967}. It is usually formulated as an exterior free boundary problem for the Laplace equation in $\mathbb{R}^3$ with boundary conditions corresponding to the type of observation. In the formulation introduced by Molodensky \cite{Molodensky1, Molodensky2}, the gravitational potential $W$ and field $G$ are prescribed on an unknown boundary diffeomorphic to the $2$--sphere by a map $\varphi : S^2 \rightarrow \mathbb{R}^3$. With the advent of satellite technologies to determine the earth's surface, high-precision studies combine satellite data with local gravity measurements. \\

The mathematical analysis of Molodensky's problem was initiated by H\"{o}rmander \cite{Hoermander}, who investigated the local existence and uniqueness of solutions based on the implicit function theorem for $C^\infty$--functions on the boundary. In this article, we consider the mathematically justified numerical analysis of the Molodensky problem, using a feasible reformulation of the constructive proof of existence.\\

To solve the free boundary problem, we iteratively construct a sequence $(\varphi_m)_{m \in \mathbb{N}_{0}}$ of approximate solutions, where $\varphi_m$ is obtained from the solution to the problem linearized around $\varphi_{m-1}$. As the main difficulty, the solution operator to the linearization is of order $2$. It is unbounded in the natural Banach spaces of functions, and approximations constructed from Banach fixed--point iterations will eventually lose regularity. Based on the insights of Nash \cite{Nash} and Moser \cite{Moser}, H\"{o}rmander defines a smoothened, convergent sequence of approximate solutions for data, which are close in a H\"{o}lder norm to those of a given solution.\\

Note that the increments
$\varphi_{m+1}-\varphi_m$ differ from certain increments considered in the geodesic
literature. In our case, the linearized problem involves the solution of
an exterior boundary problem for the \emph{homogeneous} Laplace equation, and no ``topographic-isostatic'' correction is necessary to correct for unknown masses between $\varphi_m(S^2)$ and $\varphi_0(S^2)$.\\

In this work, we show that a variant of
this construction is numerically feasible and obtain an a priori estimate
for the error in the $m$--th step. We test it on a simple example, in which
the algorithm is started with the canonical embedding of the unit sphere
in $\mathbb{R}^3$ and recovers the sphere of radius $1.1$. The reader will find a precise exposition of our main results below. In particular, we legitimize the use of smoothing operators based on the solution of higher--order heat equations in this context. We also numerically explore the accurate evaluation of second--order derivatives on the boundary.\\

This article shows that the rigorous numerical solution of the nonlinear Molodensky problem is computationally feasible. To improve the stability and numerical accuracy of the method proposed in this paper, further development needs to focus on the domain discretization error, comparing higher--order approximations of the surface with meshless methods for realistic data. In particular, to be relevant for the geodesic community, one might model the exact surface
by the ETOPO1 model of the earth and compute the gravity vector from the
EGM2008 model. A more realistic model problem could then try to recover
$\varphi$ starting from the GRS80 ellipsoid as surface $\varphi_0$ and the
corresponding Somigliani-Pizetti field as $(W_0, G_0)$ \cite{Graf}.\\

For simplicity, most of this article considers the case of a nonrotating earth. The analysis, however, readily extends to the general case.

\subsection{Related work}

Local existence for the Molodensky problem with data in a neighborhood of a given solution consisting of a gravitational potential $W_0$, gravitational field $G_0$ and surface $\varphi_0$, regularity and uniqueness were first established by H\"{o}rmander.
\begin{theorem}[Theorem 3.4.1, \cite{Hoermander}] \label{nonlinmoloexunq}
Assume that $(W_0, G_0, \varphi_0)$ satisfies the assumptions in Section \ref{ch:algo}, and let $\epsilon >0$.\\
a) For all $W,G$ in an  $\mathscr {H}^{2+\epsilon}$ neighborhood of $W_0,G_0$, the Molodensky problem admits a solution $\varphi$ close to $\varphi_0$ in $\mathscr {H}^{2+\epsilon}$. \\
b) If  $W,G$ are in $\mathscr {H}^a$ for some non-integer $a>2+\epsilon$, then $\varphi \in \mathscr {H}^a$.\\
c) Any small $\mathscr {H}^{3+\epsilon}$ neighborhood of $\varphi_0$ contains at most one solution of the problem.
\end{theorem}
See the appendix for our definition of the space $\mathscr {H}^a$ of H\"{o}lder--continuous functions. Based on ideas of Nash and Moser, H\"{o}rmander found an iterative method of solution, which overcomes the loss of derivatives in each iteration step with an abstract smoothing operator. We consider the numerical aspects of his approach and discuss an implementation using the boundary element method (BEM).\\
The numerical solution of the linearized Molodensky problem using linear boundary elements has first been analyzed by Klees et al.~\cite{Schwabfast}, see also work by Holota \cite{Holota} or Freeden and coauthors \cite{FreedenMayer06}. Their general convergence analysis implies the convergence and stability of our solution, and we refer to their paper for a discussion of the linearized problem. \\
To overcome a loss of derivates in simple ordinary differential equations, smoothed iterative solvers have been investigated by Jerome \cite{Jerome} and Jerome and Fasshauer \cite{Jeromeheatsm}. The smoothing step would usually involve the solution of a heat equation, which does not satisfy the non-saturation property made in previous theoretical analyses (Property \ref{propersmoothlikehoer} (iii) in Section \ref{ch:smoothing}). However, smoothing with the help of the heat equation is used in a variety of contexts
such as high-dimensional statistics or image processing, and numerous optimized implementations are readily available. Our analysis rigorously legitimizes their application. We are not aware of previous numerical investigations of smoothed fixed--point iterations for $2$-- or higher--dimensional problems.\\
There has been recent interest in the numerical analysis of semilinear elliptic and parabolic equations on the sphere (see e.g.~\cite{Wendland}). Our article focuses mostly on the specific difficulties of the Molodensky problem.\\

\subsection{Formulation of the problem}

In the following, the earth is assumed to be a rigid body of mass $M$, which rotates with a constant and known angular velocity $\omega$ around the $x_3$ axis through the center of mass $0$.
The surface is diffeomorphic to the sphere $S^2=\{x \in \mathbb R^3; x_1^2+x_2^2+x_3^2=1\}$ under a map $\varphi : S^2 \to \mathbb R^3$.
The measured data $W$ and $G$ may then be considered as functions on $S^2$
\begin{equation*}
W:S^2 \rightarrow \mathbb R,\quad G:S^2 \rightarrow \mathbb R^3.
\end{equation*}
The Molodensky problem is to find a sufficiently smooth embedding $\varphi : S^2 \rightarrow \mathbb R^3$ such that
\begin{equation}
 W = w \circ \varphi\ , \quad G = \nabla w\circ \varphi= g \circ \varphi \quad \mbox{on} \quad S^2, \label{nonlinmol}
\end{equation}
where  $w : \varphi( S^2) \to \mathbb R$ denotes the gravity potential, $g = \nabla w : \varphi(S^2) \to \mathbb R^3$ the gravity vector.
The static gravitational potential $v$ is harmonic in the exterior of the earth with boundary values
\begin{equation}
w(x) = v(x) + \frac{\omega^2}{2}(x_1^2 + x_2^2)\quad \text{ on }\quad \varphi(S^2) \ . \label{potential}
\end{equation}
The center of mass is fixed to $0$ by the radiation condition
\begin{equation}
v(x) = \frac{M}{|x|} + \mathcal{O}(|x|^{-3})\quad \mbox{for} \quad x \rightarrow \infty \ . \label{radiation}
\end{equation}
We also require that the Marussi condition \cite{Krarup},
\begin{equation}\label{Marussi}
\det g'(x) \neq 0, \quad x\in \varphi(S^2)\ ,
\end{equation}
is fulfilled and set
\begin{equation}\label{intromolo:h}
h=-{g'}^{-1}g \ .
\end{equation}
To determine $\varphi$,  we start with a $C^\infty$ solution of the Molodensky problem $(\varphi_0, W_0, G_0)$  such that (\ref{Marussi}) is satisfied and  the corresponding $h_0$ is never tangential to $\varphi_0(S^2)$.
Linearization around $(\varphi_0, W_0, G_0)$ results in the oblique Robin problem
\begin{equation*} \label{linearizedproblem}
\Delta u = 0 \quad \mbox{ outside } \varphi_0(S^2)\ , \ (u+\langle \mbox{grad}\,u,h_0\rangle) \circ\varphi=f \quad \mbox{on} \quad S^2\ , \ u(x) = \frac{M}{|x|} + \mathcal{O}(|x|^{-3})\ .
\end{equation*}

Under certain conditions on $\varphi$ and $h$, which we assume to hold in what follows, the oblique Robin problem is a regular elliptic boundary value problem, and Fredholm's alternative holds. The oblique Robin problem can be shown to have Fredholm index $0$, and the homogeneous problem admits three linearly independent eigensolutions. Existence and uniqueness are therefore assured provided that $f$ belongs to a subspace of $C^\infty({\partial \Omega})$ of codimension $3$.\\

For example, the linearized problem is well--posed when $\varphi_0(S^2)$ is the unit sphere and $W_0, G_0$ are the Newton potential and its gradient. The three--dimensional subspace is then spanned by the spherical harmonics $\{Y_{1,-1}, Y_{1,0}, Y_{1,1}\}$ of degree $1$. A careful analysis in \cite{Hoermander}, Theorem 1.5.1 and subsequent examples, shows that the problem remains well--posed as long as $\varphi$ is close to the identity and $h$ makes a small angle with both the exterior normal vector field and the radial vector field $\frac{x}{|x|}$. Numerically, the conditions are satisfies e.g.~for any topography with slopes $< 41^\circ$, provided slopes $>10^\circ$ are sufficiently rare.
In general, we may then use the restrictions $\{A_j \}_{j=1}^3$ to $\varphi_0(S^2)$ of harmonic functions $u_j^{\varphi_0}$ with $u_j^{\varphi_0}(x) \to 0$  for $|x|\to \infty$,
such that the first degree harmonics in the multipole expansion at infinity are linearly independent.\\

The augmented formulation of the linearized problem now reads as follows:
Find $u$ and constants $a_{j} \in \mathbb R$ such that
\begin{align}
&\Delta u=0 \quad \mbox{outside} \quad \varphi(S^2),  \nonumber \\
& u + \nabla u \cdot h = f -\sum_{j=1 }^3 a_{j} A_j  \quad \mbox{on} \quad \varphi(S^2)  \label{intro:linpro}  \\
& u(x)=\frac{c}{|x|} + O(|x|^{-3}) \quad \mbox{when} \quad |x| \rightarrow \infty,\quad c \in \mathbb R, \nonumber
\end{align}
for suitable $f \in C^{\infty}(\varphi(S^2))$.\\

The iterative solution of the Molodensky problem solves a sequence of linearized problems with
\begin{equation} \label{intro4}
f_m =  \dot W_m \circ \varphi_m^{-1}+(\dot G_m \circ \varphi_m^{-1}) \cdot h_m
\end{equation}
as specified in Section \ref{ch:algo} and updates $\varphi$ by an increment proportional to $\dot{\varphi} = (\nabla g \circ \varphi)^{-1} (\dot{G} - \nabla u \circ \varphi)$. As the main difficulty of the Molodensky problem, the update $\dot{\varphi}$ is in general less regular than ${\varphi}$.

\subsection{Error estimates}

We complement H\"{o}rmander's existence result by giving an a priori estimate for the error between the solution and the $m$-th iterative approximation. Here and in the remainder of the article, $\epsilon>0$ will be such that the relevant H\"{o}lder exponent is not an integer.\\

A sequence of approximate solutions $(W_m, \varphi_m)$ will be defined in Section \ref{ch:algo}, depending on a given solution $(W_0, G_0, \varphi_0)$ and two parameters $\theta_0$ and $\kappa$. Under the assumptions of this section we obtain the following a priori estimate in H\"{o}lder norms for data $(W,G)$ in an $\mathcal{H}^{2+\epsilon}$ neighborhood of $(W_0, G_0)$:

\begin{theorem}\label{apriorithm}
Let $\alpha>2+2\epsilon$, $0<a <\alpha$, $E=a- \alpha -1$ and $\tau>0$ sufficiently small. Then there exist constants $\theta_0, \kappa_0>0$ and $C_\tau>0$ such that the approximate solutions $(W_m, \varphi_m)$ satisfy for all $m \geq 0$ and $\kappa \geq \kappa_0$
\begin{equation}
\|W-W_m\|_{a+\epsilon} +\|\varphi-\varphi_m\|_{a+\epsilon}\leq C_\tau \left(\|W-W_0\|_{\alpha+\epsilon}+\|G-G_0\|_{\alpha+\epsilon}\right)(\theta_0^\kappa+m)^{\frac{E+1+\tau}{\kappa}}
\end{equation}
The constants depend only on the data.
\end{theorem}

Below, we note that the estimate on $W-W_m$ and $\varphi-\varphi_m$ implies a corresponding estimate on $G-G_m$.

The explicit dependence on the parameters $\theta_0$ and $\kappa$ obtained in Theorem \ref{apriorithm} allows us to investigate also a restarted algorithm, which is less susceptible to the spreading of numerical errors (see Algorithm \ref{alg:nashhormsmrest}):
\begin{itemize}
 \item[(0)] Choose an approximate solution $(W_0, G_0,\varphi_0)$ and $\theta_0$.
 \item[(1)] Using $(W_0, G_0,\phi_0)$ do $k$ steps of H\"ormander's method leading to $(W_k, G_k,\varphi_k)$.
 \item[(2)] Set $(W_0, G_0,\phi_0) = (W_k, G_k,\varphi_k)$, a corresponding new $\theta_0$ and go to (1).
\end{itemize}
We denote the approximate solution after $l$ iterations of (1) by $(W^{(l)}, G^{(l)},\varphi^{(l)})$ and the corresponding $\theta_0$ by $\theta_{0,l}$. From the Lipschitz--continuity of the map $\Gamma: (W,\varphi) \mapsto G$,
\begin{equation*}
\|G-G^{(l-1)}\|_{\alpha +\epsilon} \leq C_K \left(\|W-W^{(l-1)}\|_{\alpha+2+\epsilon}+\|\varphi-\varphi^{(l-1)}\|_{\alpha+2+\epsilon}\right)
\end{equation*}
on bounded subsets $\|\varphi-\varphi_0\|_{\alpha+2+\epsilon} \leq K$, we obtain
\begin{align*}
\|W-W^{(l)}\|_{a+\epsilon} +\|\varphi-\varphi^{(l)}\|_{a+\epsilon}&\leq C_\tau \left(\|W-W^{(l-1)}\|_{\alpha+\epsilon}+\|G-G^{(l-1)}\|_{\alpha+\epsilon}\right)(\theta_{0,l-1}^\kappa+k)^{\frac{E+1+\tau}{\kappa}}\\
&\leq C_\tau \left(\|W-W^{(l-1)}\|_{\alpha+\epsilon}+\|G-G^{(l-1)}\|_{\alpha+\epsilon}\right)(\theta_{0,l-1}^\kappa+k)^{\frac{E+1+\tau}{\kappa}}\\
&\lesssim \left(\|W-W^{(l-1)}\|_{\alpha+2+\epsilon} +\|\varphi-\varphi^{(l-1)}\|_{\alpha+2+\epsilon}\right) (\theta_{0,l-1}^\kappa+k)^{\frac{E+1+\tau}{\kappa}}
\end{align*}
for any $\tau>0$.\\

Iterating this estimate yields that the sequence of iterates of the restarted algorithm converges, if we choose a sufficiently rapidly increasing sequence of
$\theta_{0,l}$. As the proof will show, $\theta_{0,l}$ has to be chosen as a function of certain higher H\"older norms of $W-W^{(l)}$ and $\varphi-\varphi^{(l)}$. Summing up:

\begin{proposition}\label{restart} Under the assumptions of the theorem,
\begin{equation}\label{restarteq}
\|W-W^{(l)}\|_{a+\epsilon} +\|\varphi-\varphi^{(l)}\|_{a+\epsilon}\lesssim \left(\|W-W^{(l-1)}\|_{\alpha+2+\epsilon} +\|\varphi-\varphi^{(l-1)}\|_{\alpha+2+\epsilon}\right) (\theta_{0,l-1}^\kappa+k)^{\frac{E+1+\tau}{\kappa}} .
\end{equation}
In particular, the algorithm with restart converges for appropriate $\theta_{0,l}$ depending on $W-W_0$ and $\varphi-\varphi_0$.
\end{proposition}

Iterating \eqref{restarteq}, we can estimate $W-W^{(l)}$ and $\varphi-\varphi^{(l)}$ in terms of the initial error $W-W_0$, $\varphi-\varphi_0$.\\

The proofs of Theorem \ref{apriorithm} and Proposition \ref{restart} are contained in the appendix, Section \ref{app:proofs}.

\subsection{Numerical considerations}

From a computational perspective, we give a proof-of-principle for fully nonlinear numerical computations with the Molodensky problem. Besides showing the practicality of standard heat equation smoothing, the implementation involves subproblems of possible interest beyond the particular problem: \\

a) The updates of $\varphi$ and $h$ require the highly accurate evaluation on the surface for first and second derivatives of the gravitational field, expressed as derivatives of the single layer potential \cite{Schloemerkemper, c3, Schwab}. We show that second order finite elements are necessary. If we combine them with a partially analytic evaluation of the singular integrals for the gradient and finite differences for the Hessian, the error becomes negligible.\\

b) Surface discretization turns out to be the dominant source of error for smooth data. This shows the need for third or higher--order approximations or meshless methods. See e.g.~\cite{SauterSchwab}, Ch.~8, for boundary integral equations on surfaces approximated to degree $p$. The approximately spherical geometry suggests meshless eigenfunction expansions, and \cite{StCoTr1} contains a first study of in the case of a particular linearized problem.\\

Parallelization and nontrivial optimization are necessary to deal with the large, dense matrices of a higher-order boundary element formulation and the large number of data points available from geodesic measurements. For the linearized Molodensky problem, panel clustering and fast multipole methods have been discussed in \cite{Schwabfast}, particularly in the case of a non-oblique Robin problem.

\section{Iterative solution of the nonlinear problem}\label{ch:algo}
Our approach to the Molodensky problem is based on the abstract Nash--H\"ormander iteration \cite{Hoermander}, as adapted to the Molodensky
problem. It solves a sequence of linearized problems, whose solutions are less regular than the data, and recovers regularity with the help of a smoothing operator. To simplify notation, we restrict ourselves to a nonrotating earth, $\omega = 0$. The equivalence of the formulation to the standard Nash--Moser iteration is discussed in Appendix \ref{app:proofs}. See also \cite{costea2012} for a detailed derivation.\\

We fix a $C^\infty$ solution of the Molodensky problem $(\varphi_0, W_0, G_0)$  with a gravitational potential $v_0$ which satisfies the Marussi condition. We  assume that  the corresponding $h_0$ is not tangential to $\varphi_0(S^2)$ and that the homogeneous linearized problem (\ref{linearizedproblem}) has only the  trivial solution.\\

As stated in the introduction, the linearized Molodensky problem involves an oblique Robin problem in the exterior $\mathbb{R}^3 \backslash\bar \Omega_m$ of $\varphi_m(S^2)$. In each iteration step $m$, we solve the augmented problem:\\
Given $\dot W_m : S^2 \rightarrow \mathbb R, \dot G_m : S^2 \rightarrow \mathbb R^3, h_m:\varphi_m(S^2) \rightarrow \mathbb R^3 \quad \mbox{and} \quad \varphi_m:S^2  \rightarrow \varphi_m(S^2) \subset \mathbb R^3$,
find $u_m: \mathbb R^3\backslash \bar \Omega_m \rightarrow \mathbb R$ and constants $a_{j,m} \in \mathbb R$ such that
\begin{align}
&\Delta u_m=0 \quad \mbox{in} \quad \mathbb R^3\backslash \bar \Omega_m, \, \nonumber \\
& u_m + \nabla u_m \cdot h_m =f_m -\sum_{j=1 }^3  a_{j,m} \widetilde A_j (x)  \quad \mbox{on} \quad \varphi_m(S^2)  \label{stronglin}  \\
& u_m(x)=\frac{c}{|x|} + O(|x|^{-3}) \quad \mbox{when} \quad |x| \rightarrow \infty,\quad c \in \mathbb R \ . \nonumber
\end{align}
Again $f_m=\dot W_m \circ \varphi_m^{-1}+(\dot G_m \circ \varphi_m^{-1}) \cdot h_m$. A feasible choice for $\widetilde A_j$ is discussed after (\ref{definitionofVK}).\\ 

The solution $u_m$ determines a nonlinear correction to $\varphi_m$, $\dot \varphi_m$, given by
\begin{equation}
 \dot \varphi_m = (\nabla g_m\circ\varphi_m)^{-1}(\dot{G}_m-\nabla u_m \circ \varphi_m)\ . \label{corrphi}
\end{equation}
The gravitational potential $g_m$ is determined from the approximation to the potential as computed in the first $m$ steps,
 \begin{equation}
 w_m=W_{m-1} \circ \varphi_m^{-1} +\triangle_m u_m \quad \mbox{on} \quad \varphi_m(S^2)\ , \label{wmtotal}
\end{equation}
$$W_{m-1}=\begin{cases}v_0 \circ \varphi_0 + \triangle_0 u_0 \circ \varphi_0 + \triangle_1 u_1 \circ \varphi_1 + \dots + \triangle_{m-1} u_{m-1} \circ \varphi_{m-1},\! \quad\mbox {for} \quad\!m\geq 1\\
v_0 \circ \varphi_0  \quad\mbox {for} \quad m=0 \end{cases}  \ ,$$
for suitable stepsizes $\Delta_j$ and initial approximation $v_0$ by solving an exterior Dirichlet problem:\\
For given $w_m$ on $\varphi_m(S^2)$, find $ \overline {v}_m : \mathbb R^3\backslash \bar \Omega_m \rightarrow \mathbb R$ and constants $a_{j,m} \in \mathbb R$ such that
\begin{align}
&\Delta \overline {v}_m =0 \quad \mbox{in} \quad \mathbb R^3\backslash \bar \Omega_m, \nonumber\\
& {\overline{v}_m}_{|_{\partial {\Omega_m}}} = w_m -\sum_{j=1}^3 a_{j,m} \widetilde A_j (x)\big|_{x \in\:\varphi_m(S^2)}  \quad \mbox {on} \quad \varphi_m(S^2) \label{Dirichlet} \\
& \overline{v}_m(x)= \frac {c}{|x|} + O(|x|^{-3}) \quad \mbox {when} \quad  |x| \rightarrow \infty \ .\nonumber
\end{align}
Here  $\widetilde A_j$ need not be the same as in (\ref{stronglin}). For the numerical solution using boundary elements, we can reuse the matrix entries of  the discretized Robin problem for this Dirichlet problem.\\
Now equation (\ref{corrphi}) yields the surface update $\dot \varphi_m$ to $\varphi_m$ using $g_m=\nabla \overline{v}_m$ and $\nabla g_m=\nabla^2 \overline{v}_m$.\\

The full iterative method involves smoothing in each step based on the solution operator $S_{\theta}$ to a higher-order heat equation as discussed in Section \ref{ch:smoothing}. It reads as follows:
\begin{alg} \label{alg:nashhormsm}
\vspace*{10mm} \hrule
\vspace*{0.5mm} (Nash-H\"ormander algorithm) \vspace*{2mm}
\hrule
\begin{enumerate}
	\item For given measured data $W, G$, choose $W_0, G_0,h_0,\varphi_0, \theta_0 \gg 1, \kappa \gg 1$
  \item For $m=0,1,2,\ldots$ do
  \begin{enumerate}
        \item Compute \begin{equation}\label{thetatocompute}
                                      \theta_m=(\theta_0^{\kappa}+m)^{1/\kappa},\quad \triangle_m=\theta_{m+1}-\theta_{m}
                                    \end{equation}
         \item Compute
                 \begin{align}\label{wreccurencetosmooth}
                  \dot {\widetilde W}_0:&=S_{\theta_0}\dot W_0=S_{\theta_0}\big(\frac{W-W_0}{\triangle_0} \big) \nonumber\\
                  \dot {\widetilde W}_m:&=\frac{1}{\triangle_m}\big(S_{\theta_m}\!(W\!-W_0\! )\!-\!S_{\theta_{m-1}}\!(W \!- W_{0})\big)
                 \end{align}
        \item Compute
                 \begin{align}\label{greccurencetosmooth}
                  \dot {\widetilde G}_0:&=S_{\theta_0}\dot G_0=S_{\theta_0}\big(\frac{G-G_0}{\triangle_0} \big) \nonumber\\
                  \dot {\widetilde G}_m:&=\frac{1}{\triangle_m}\big(S_{\theta_m}\!(G\!-G_m\!+\!\! \sum_{j=0}^{m-1}\!\!\triangle_j \dot {\widetilde G}_j )\!-\!S_{\theta_{m-1}}\!(G \!- G_{m-1} +\!\! \sum_{j=0}^{m-2}\!\!\triangle_j \dot {\widetilde G}_j)\big)
                 \end{align}
        \item Find $u_m$ by solving the linearized problem (\ref{stronglin}) with $(\dot W_m, \dot G_m)$ replaced by $(\dot {\widetilde W}_m, \dot{\widetilde G}_m)$
        \item Find $\overline v_m$ by solving (\ref{Dirichlet}) with $w_m$ as defined in  (\ref{wmtotal})
        \item Compute $g_m=\nabla \overline {v}_m$ and $\nabla g_m=\nabla^2 \overline {v}_m$
        \item Compute the surface increment $\dot \varphi_m$ by $$\dot \varphi_m= (\nabla g_m \circ \varphi_m)^{-1} (\dot{\widetilde G}_m- \nabla u_m \circ \varphi_m)$$
                  and update surface map by $\varphi_{m+1}=\varphi_m + \triangle_m \dot \varphi_m$
        \item Update direction vector and gravity potential by
     \begin{align*}
      h_{m+1}&=((-(\nabla g_m)^{-1} g_m) \circ \varphi_m) \circ ({\varphi_{m+1}})^{-1} \\
      G_{m+1}&=g_m \circ \varphi_m
     \end{align*}
      \item Stop if $\left\| g_m\circ \varphi_m-G \right\| +\left\| \overline v_m\circ \varphi_m-W \right\|< \text{tol}$
    \end{enumerate}
\end{enumerate}
\hrule
\end{alg}
$\|\cdot\|$ might usually be chosen to be e.g.~an $\mathcal{H}^a$--norm.\\

We also consider a variant of the algorithm which is restarted every $k$ steps, using as initial condition the approximate solution from the $k$-th step:

\begin{alg} \label{alg:nashhormsmrest}
\vspace*{10mm} \hrule
\vspace*{0.5mm} (Nash-H\"ormander algorithm with restart) \vspace*{2mm}
\hrule
\begin{enumerate}
      \item For given measured data $W, G$ and $k \in \mathbb{N}$, choose $W_0, G_0,h_0,\varphi_0, \theta_0 \gg 1, \kappa \gg 1$

     \item Compute $W_k, G_k, \varphi_k$ in Algorithm \ref{alg:nashhormsm}

      \item Stop if $\left\| G_k-G \right\| +\left\| W_k-W \right\|< \text{tol}$

 \item Else set $W_0= W_k, G_0=G_k, h_0 = h_k,\varphi_0=\varphi_k$, choose $\kappa, \theta_0$, and go to 2
\end{enumerate}
\hrule
\end{alg}
\vspace{0.5cm}

To solve the homogeneous exterior Robin and Dirichlet problems (\ref{stronglin}) resp.~(\ref{Dirichlet}), we use  a single layer potential ansatz for $u_m=V_m \mu_m$ and satisfy the decay condition at $\infty$ in a weak sense.
Boundary element formulations for the oblique Robin problem (\ref{stronglin}) with general boundary data were first analyzed in  \cite{Schwabfast}. They, in particular, showed the relevance of multipole expansions and panel clustering
for the fast solution for experimental geodesic data.\\

The ansatz transforms (\ref{stronglin}) into a saddle point problem for the integral operator
\begin{equation}\label{definitionofS}
\mathcal{S}=V+\frac{1}{2}\cos (\measuredangle (\boldsymbol n, \boldsymbol h))I + K'(\boldsymbol h) ,
\end{equation}
on $\varphi_m(S^2)$, which is defined in terms of the multilayer potentials
\begin{equation}\label{definitionofVK}
V \mu(x) =  \int_\Gamma \frac{\mu(y)}{4 \pi |x-y|} \;ds_y\ , \quad K^\prime(\boldsymbol h) \mu(x) = \boldsymbol h \cdot \nabla \int_\Gamma \frac{\mu(y)}{4 \pi |x-y|}   \;ds_y \ .
\end{equation}
In particular for $\boldsymbol h=\boldsymbol n$ the unit normal vector, $K'(\boldsymbol n)$ is the standard adjoint double layer potential. \\

Solving the Robin problem for arbitrary $\dot W, \dot G$ requires an appropriate choice of $\widetilde A_j$. Let $A_j=\frac{x_j}{|x|^3}$ and  $\mathcal N := \mbox{span}\,\lbrace A_j\rbrace_{j=1,\dots,3}$.
Then $u_m=V_m \mu_m$ satisfies the decay condition
of  (\ref{stronglin}) whenever $\mu_m \in L^2(\varphi_m(S^2)) \cap \mathcal N^{\perp}$. Since $f_m =  \dot W_m \circ \varphi_m^{-1}+(\dot G_m \circ \varphi_m^{-1}) \cdot h_m\in L^2(\varphi_m(S^2))$,
but not necessarily in $S( L^2(\varphi_m(S^2)) \cap \mathcal N^{\perp})$, $\widetilde A_j$ must be chosen such that
$\mbox{span}\,\lbrace \widetilde A_j \rbrace_{j=1}^3 + S(L^2(\varphi_m(S^2)) \cap \mathcal N^{\perp})=L^2(\varphi_m(S^2))$ for (\ref{stronglin}) to be well defined. $\widetilde A_j:= S A_j|_{\varphi_m(S^2)}$ is a feasible choice
and leads to an equivalent variational formulation of (\ref{stronglin}):\\
Find $(\mu_m,a_m) \in L^2(\varphi_m(S^2)) \times \mathbb R^3$ such that
  \begin{equation}\label{Wadbif:weakform1}
   \begin{alignedat}{2}
        \langle \mathcal{S} \mu_m, \phi \rangle_{\varphi_m(S^2)}+\langle \mathcal{S}  \sum_{j=1}^3 a_{j,m} A_j, \phi \rangle_{\varphi_m(S^2)} &=\langle f_m,\phi \rangle_{\varphi_m(S^2)} \quad &\forall \phi \in L^2(\varphi_m(S^2))\\
        \langle \mu_m,A_k \rangle_{\varphi_m(S^2)} &=0  \quad &\forall k \in \{1,2,3\}\ .
      \end{alignedat}
   \end{equation}
Note that any $L^2$--solution of the weak problem is actually $C^\infty$. Indeed, $\mathcal{S}$ is an elliptic pseudodifferential operator, so that $\mu_m \in L^2$ and $\mathcal{S} \mu_m = f_m - \mathcal{S}  \sum_{j=1}^3 a_{j,m} A_j \in C^\infty$ in the distributional sense implies $\mu_m \in C^\infty$ \cite{Seeley}.\\

The Dirichlet problem (\ref{Dirichlet}) is similarly reformulated:
Find $(\tilde {\mu}_m, {\widetilde a}_m) \in H^{-1/2}(\varphi_m(S^2)) \times  \mathbb R^3$ such that
  \begin{equation}\label{Wadbif:weakform2}
   \begin{alignedat}{2}
        \langle V \tilde{\mu}_m, \xi \rangle_{\varphi_m(S^2)}+\langle V  \sum_{j=1}^3 {\widetilde a}_{j,m} A_j, \xi \rangle_{\varphi_m(S^2)} &=\langle {w_m},\xi \rangle_{\varphi_m(S^2)}  \forall \xi \in H^{-1/2}(\varphi_m(S^2))\\
         \langle \tilde{\mu}_m,A_k \rangle_{\varphi_m(S^2)} &=0   &\forall k \in \{1,2,3\}\ .
      \end{alignedat}
   \end{equation}
On the right hand side,
\begin{align}\label{rhsdirichletimp}
\langle w_m,\xi \rangle_{\varphi_m(S^2)} =& \langle W_{m-1} \circ \varphi_m^{-1},\xi \rangle_{\varphi_m(S^2)} + \triangle_m \langle u_m,\xi\rangle_{\varphi_m(S^2)} \nonumber \\
=&\langle (v_0 \circ \varphi_0)\circ \varphi_m^{-1},\xi  \rangle_{\varphi_m(S^2)} + \sum_{i=0}^{m-1}  \triangle_i \langle (V_i \mu_i \circ \varphi_i) \circ \varphi_m^{-1},\xi \rangle_{\varphi_m(S^2)}\nonumber\\
&+\triangle_m \langle V_m \mu_m,\xi \rangle_{\varphi_m(S^2)}\nonumber.
\end{align}
As discussed in Section \ref{ch:nummeth}, we use standard boundary element methods for the efficient numerical solution of (\ref{Wadbif:weakform1}) and (\ref{Dirichlet}).
Note that the matrix elements of $V$ have already been computed for the Robin problem.

\section{Smoothing operator}\label{ch:smoothing}

The smoothing operators $S_\theta$ appearing in the iterative solution of the Molodensky problem compensate for the loss of derivatives in the increments of $\varphi$.
For a theoretical analysis, smoothing operators $S_\theta$ defined from compactly supported functions $\phi \in C_0^{\infty}(\mathbb R)$ are most convenient. If $M$ is a submanifold of $\mathbb{R}^3$, we define
$S_\theta^{th}: \mathscr{H}^a(M) \to \mathscr{H}^a(M)$ as
\begin{equation*}
 S_\theta^{th} u := \phi(\frac{1}{\theta^2} \Delta_M) u\ .
\end{equation*}
Here $\Delta_M$ is the Laplace--Beltrami operator associated to the metric which $M$ inherits from $\mathbb{R}^3$, and $\phi(\frac{1}{\theta^2} \Delta_M) u$ can be computed by expanding $u$ into $L^2(M)$--eigenfunctions of $\Delta_M$.\\

These operators have the following properties (Theorem A.10, \cite{Hoermander}):
\begin{proper}\label{propersmoothlikehoer} For all $u \in C^{\infty}(M)$ we have
  \begin{itemize}
  \item[(0)] $\|S_\theta^{th} u -u\|_a \stackrel{\theta \rightarrow \infty }{\longrightarrow} 0 ;$
   \item[(i)] $\|S_\theta^{th} u\|_b \leq C \|u\|_a, \quad b \leq a ;$
    \item[(ii)] $\|S_\theta^{th} u\|_b \leq C \theta^{b-a}\|u\|_a,\quad a \leq b ;$
     \item[(iii)] $\|u-S_\theta^{th} u\|_b \leq C \theta^{b-a}\|u\|_a,\quad b \leq a ;$
     \item[(iv)] $\bigg\|\frac{d}{d\theta} S_\theta^{th} u\bigg\|_b \leq  C \theta^{b-a-1}\|u\|_a$.
 \end{itemize}
\end{proper}
The oscillatory nature of the corresponding integral kernels renders a stable implementation of these operators difficult. In practice, solution operators to heat-like equations are frequently used as smoothing operators in a variety of contexts such as high-dimensional statistics, image
processing etc.~with readily available, optimized implementations. They have been also used in the numerical explorations of smoothed Newton and similar methods, though without much justification, see e.g \cite{Jeromeheatsm, Jerome}.\\
More generally than the heat equation, we consider smoothing operators associated to $\phi(x)=e^{-|x|^{k}}$, $k \in \mathbb{N}$. They correspond to the time--$1/\theta^{2k}$ solution of the $k$--harmonic heat equation
\begin{alignat*}{2}
\frac{d}{d t} v(x,t) - A v(x,t)&=0& &\,\,\,\mbox{in} \quad M \times (0, \infty) \\
v(x,0) &=u(x)&   &\,\,\,\mbox{in} \quad  M
\end{alignat*}
with $A:=(-1)^{k+1} \Delta_M^k$ and $u \in \mathscr {H}^a$. Considering $A:\mathscr {H}^{a+2k}\subset \mathscr {H}^{a} \rightarrow \mathscr {H}^{a}$ as an unbounded operator on the H\"older spaces ($a>0, a \notin \mathbb N$) we have the following theorem.
\begin{theorem}\label{heatthem}
$A$ generates an analytic semigroup $e^{tA}$ on $\mathscr {H}^a$, and the operator $S_\theta:=e^{\frac{1}{\theta^{2k}}A}$ satisfies the properties (0), (i), (ii), (iv) (with $S_\theta^{th}$ replaced by $S_\theta$) and in addition
\begin{equation*}
(iii') \qquad  \|u- S_\theta u\|_b \leq C \theta^{b-a} \|u\|_a, \quad \forall\:0 \leq a-b < 2k.
\end{equation*}
\end{theorem}
A proof of this theorem will be given in the Appendix. In particular, we note that these restricted properties are sufficient to rigorously analyze the convergence of the algorithms presented in Section \ref{ch:algo}.

\section{Numerical methods} \label{ch:nummeth}

The main effort to numerically solve the Molodensky problem consists in the accurate solution of the saddle point problems (\ref{Wadbif:weakform1}) and (\ref{Wadbif:weakform2}) for the multilayer potentials $\mathcal{S}$ and $V$ on
the surfaces $\varphi_m(S^2)$.

For the solution, we choose a quasi--uniform triangulation $\mathcal T_h^m$ of $\varphi_m(S^2)$ by plane triangles of diameter $\sim h$ and solve the saddle point problems for piecewise polynomial elements on the resulting discretized surface.

In the case of (\ref{Wadbif:weakform1}), the discretized formulation for the subspace $$S_{h,m}^p=\{\mbox{space of discontinuous piecewise polynomials of degree $p$ on } \mathcal T_h^m \}=\mbox{span} \{b_j \}_{j=1}^N$$ of $L^2(\mathcal T_h^m)$ reads as follows:
Find $(\mu_{m,h}, a_{m,h}) \in S_{h,m}^p \times \mathbb R^3$ such that
  \begin{equation}\label{Wadbif:discreteform1}
   \begin{alignedat}{2}
        \langle \mathcal{S} \mu_{m,h}, \phi_h \rangle_{\mathcal T_h^m}+  \sum_{j=1}^3 a_{j,m,h} \langle \mathcal{S}A_j, \phi_h \rangle_{\mathcal T_h^m} &=\langle f_m,\phi_h \rangle_{\mathcal T_h^m}  \quad \forall \phi_h \in S_{h,m}^p\\
        \langle \mu_{m,h},A_k \rangle_{\mathcal T_h^m} &=0 \qquad \forall k \in\{1,2,3\}
      \end{alignedat}
   \end{equation}
Expanding $\mu_{m,h}$ and $\phi_{m,h}$ in the basis $b_j$,
\begin{equation*}
\mu_{m,h} =\sum_{j=1}^N \mu_{j,m} b_j, \quad \phi_{m,h}
=\sum_{j=1}^N \phi_{j,m} b_j \ ,
\end{equation*}
we obtain a matrix equation of the form:
\[
\begin{bmatrix}  \mathcal{S} & \widetilde{\mathcal{S}} \\ \varLambda & 0 \end{bmatrix} \begin{bmatrix} \vec {\mu}_h \\ \vec {a}_h \end{bmatrix}=\begin{bmatrix} \vec f \\ \vec { 0} \end{bmatrix}
\]
where $(\mathcal{S}_{kj})= \langle \mathcal{S} b_j,b_k \rangle,\ (\widetilde{\mathcal{S}_{kj}})= \langle \mathcal{S} A_j, b_k \rangle, (\varLambda_{kj})=\langle b_j,A_k \rangle, (f_k)= \langle f_m,b_k \rangle$ and $ \vec \mu \in \mathbb R^{N}, \vec a \in \mathbb R^3$.\\

For the Dirichlet problem (\ref{Dirichlet}) we consider $S_{h,m}^p$ as a subspace of $H^{-1/2}(\mathcal T_h^m)$, leading to the problem:
Find $(\tilde {\mu}_{m,h},  {\widetilde a}_{m,h}) \in S_{h,m}^p \times \mathbb R^3$ such that
  \begin{equation}\label{Wadbif:discreteform2}
   \begin{alignedat}{2}
        \langle V \tilde{\mu}_{m,h}, \xi_h \rangle_{\mathcal T_h^m}+ \sum_{j=1}^3  {\widetilde a}_{j,m,h} \langle V A_j, \xi_h \rangle_{\mathcal T_h^m} &=\langle w_m,\xi_h \rangle_{\mathcal T_h^m}  \quad \forall \xi_h \in \mathcal T_h^m\\
      \langle \tilde{\mu}_{m,h},A_k \rangle_{\mathcal T_h^m} &=0 \qquad \qquad \forall k\in\{1,2,3\}\ .
      \end{alignedat}
   \end{equation}
In matricial form:
\[
\begin{bmatrix}  V & \widetilde V  \\ \Lambda  & 0 \end{bmatrix} \begin{bmatrix} \vec {\widetilde {\mu}}_h \\ \vec {\widetilde {a}}_h \end{bmatrix}=\begin{bmatrix} \vec{w}_m \\ \vec { 0} \end{bmatrix}.
\]

The suggested numerical solution of the linearized Molodensky and Dirichlet problems is stably convergent with quasi--optimal convergence estimates. They are derived in \cite{Schwabfast, costea2012} by verifying an inf--sup condition for the boundary element spaces.\\

Our theoretical analysis legitimizes using the solution of heat-like equations on $\varphi_m(S^2)$ as smoothing operator $S_{\theta}$ in the algorithm. We use the implementation presented in \cite{Seo:2010:HKS:1926877.1926943}, which computes an FEM-solution of the heat equation for the geometric Laplace--Beltrami operator by spectral methods. $S_\theta F$ is defined as the solution at time $\theta^{-2}$ with initial condition $F$.

\subsection{Details}
The initial sphere $ S^2$ is triangulated by starting with an icosahedral mesh and subsequent refinements \cite{01technicalmanual}. Assigning to each node a vector in $\mathbb{R}^3$ corresponds to a continuous, piecewise linear representation of $\varphi_m$. As we need to evaluate second derivatives of the gravitational potential, the polynomial degree on each triangle is $p=2$, and $h_m$ is represented by a discontinuous piecewise constant function interpolating the $h_m$ from equation (\ref{stronglin}) in the midpoints of each triangle. Furthermore, $G_m$ is the linear interpolation of $g|_{\varphi_m( S^2)}$, obtained from equation (\ref{Dirichlet}), in the nodes.
The local basis functions $b_j$ are monomials for both the linearized Molodensky problem and for the auxiliary Dirichlet problem. We use analytic expressions to compute $\langle V b_j, b_k\rangle$ and $K(h) b_k$ (\cite{MaischakTech}) and perform an hp-composite Gauss quadrature with geometrically graded meshes \cite{Erichsen,Schwab} to determine $\langle K'(h) b_j, b_k\rangle = \langle b_j, K(h) b_k\rangle$. $\langle V A_j,b_k \rangle$ and $\langle K'(h) A_j,b_k \rangle$ are treated analogously.\\

The update of $\varphi$ requires the delicate computation of the Hessian for the gravitational potential, $\nabla g_m$, on $\varphi_m(S^2)$. Our numerical experiments suggest to derive the component of $g$ in some direction $e$ using the adjoint of the double layer potential, $e \cdot g_m = \frac{e\cdot \nu}{2} \mu_{m,h} + K'(e) \mu_{m,h}$, and to use finite differences to compute $\nabla g_m$. More precisely, for the tangential components we use a central finite difference scheme, $U^\prime(x) \approx \frac{U(x+\delta)-U(x-\delta)}{2\delta}$, in the exterior domain and extrapolate to the boundary. If $x'$ is the point closest to $x$ on the discretized boundary, we choose $\delta$ smaller than the distance of $x'$ to the nearest vertex of the triangulation. For the normal component we combine the central finite difference scheme with a Crank-Nicolson method, $U^\prime(x) \approx \frac{4U(x+\delta)-3U(x)-U(x+2\delta)}{2\delta}$, which is forward oriented. The error in both schemes is of order $\delta^2$.\\

Alternative approaches are discussed in detail in Section \ref{seondderiv}.

\section{Numerical experiments}
\label{numex}

The numerical experiments were carried out on a cluster  with 5 nodes \`{a} 8 cores  with  2.93Ghz and 48GB memory,
where each core uses two Intel Nehalem X5570 processors. With parallelization and optimization we need $20$ minutes for each iterations in the case of icosahedron refinements corresponding to $320$ triangles ($N=2$), whereas we need $3$ hours for refinements corresponding to $1280$ triangles ($N=3$). A further global refinement would lead to $5120$ triangles and $16$ times larger dense matrices. For realistic models of the gravity field used in geodesy, the data is often given at grid points spaced less than 10 degrees apart, corresponding e.g.~in \cite{Schwabfast} to triangulations with between 2048 and 131072 triangles. This shows the need to develop local adaptive refinements and matrix compression. Local refinements can be introduced into the restarted Algorithm $2$ every $k$ steps.

\subsection{Computation of second derivatives}\label{seondderiv}

The accurate determination of second derivatives of the gravitational potential on the surface is crucial for our approach and is of interest also in other contexts.
The first derivatives of the single layer potential on the boundary have been analyzed in \cite{Schloemerkemper}. They can be evaluated by a composite hp Gauss quadrature
with geometrical grading towards the singularity \cite{Erichsen,Schwab}. Higher derivatives have been analyzed by Schulz, Schwab and Wendland, e.g.~in \cite{SchwabWendland, SchulzSchwabWendland}, who compute the second normal derivative from the less singular tangential derivatives in a similar way.
For the current problem, a simpler approach gives sufficient accuracy.\\
We investigate a general Dirichlet problem
\begin{subequations}
\begin{alignat}{2}
-\Delta u &= 0 &\quad &\text { in } \quad  \mathbb{R}^d \setminus \overline{\Omega} \\
u&=f &\quad &\text { on } \quad \Gamma:=\partial \Omega
\end{alignat}
\end{subequations}
with an appropriate decay condition. Using a single layer potential ansatz, the problem is equivalent to the integral equation $V\mu = f$, where
\begin{align}
u(x)=V \mu(x)=\int_\Gamma k(x,y) \mu(y) \;ds_y\ , \quad k(x,y)=\begin{cases} -\frac{1}{2\pi}\log \|x-y \|, & \quad  d=2\\ \frac{1}{4\pi} \frac{1}{\|x-y\|}, & \quad d=3. \end{cases}\label{eq:kernelDef}
\end{align}
As in Section \ref{ch:nummeth}, we solve the discretized Galerkin equations on the discretized boundary $\Gamma_h$,
\begin{align} \label{eq:discreteFormulation}
 \langle V \mu_h,\xi  \rangle_{\Gamma_h}=\langle g,\xi \rangle_{\Gamma_h} \quad \forall \: \xi \in S_{h,\Gamma}^p \ ,
 \end{align}
in the subspace $S_{h,\Gamma}^p \subset H^{-1/2}(\Gamma_h)$ of piecewise polynomials of degree $p$.
To approximate the Hessian of $u$, we have to evaluate the Hadamard finite-part integral
\begin{align} \label{eq:parfinieIntegral}
 \nabla \nabla V \mu_h(x)= \text{p.f.} \int_{\Gamma_h} \nabla_x \nabla_x k(x,y) \mu_h(y) \;ds_y \qquad (x \in \Gamma_h)\ .
\end{align}

The hypersingular kernel $\nabla_x \nabla_x k(x,y)$ is the main challenge in evaluating the potential. We analytically compute the gradient $w = \nabla V \mu_h$ using the adjoint double layer potential. The second derivative $\nabla w$
is approximated by second--order accurate finite differences (FD):
\begin{align*}
\frac{\partial w(x)}{\partial n}&=\frac{4w(x + \delta \cdot n) - 3w(x) - w(x + 2\delta \cdot n) }{2\delta} + O(\delta^2) \\
\frac{\partial w(x)}{\partial t} &=\frac{w(x+\delta \cdot t) -w(x-\delta \cdot t) }{2\delta}  + O(\delta^2)
\end{align*}

In the computations, the step size $\delta$ is set to $10^{-4}$ for the normal component and to $10^{-5}$ for the tangential component when approximating second derivatives.
For the presented numerical experiments the FD-approximation error is of magnitude $10^{-7}$ if no Galerkin-BEM  approximation error were to occur.
However, for very small step sizes the finite differences become numerically instable and the BEM-error is dominating.
If $H=(H_{ij})$ denotes the exact and $H_h=(H_{h,ij})$ the approximated Hessian, we measure the error in a point $x$ as $\left(\sum_{i,j} (H_{h,ij}(x)-H_{ij}(x))^2\right)^{1/2}$.
The BEM-approximation error is measured in the energy norm  $\|\mu-\mu_h\|_V^2:=\langle V(\mu-\mu_h),\mu-\mu_h\rangle_{\Gamma_h}$.
\begin{example}
Let $\Omega=\left[-\frac{1}{2},\frac{1}{2}\right]^2$ be the domain and $u=\ln \|x\|$ the exact solution. Then the exact Hessian is $H(x)= \frac{1}{x^2} \left(
\begin{array}[pos]{c c}
	1-2x_1^2 & -2x_1x_2 \\ -2x_1x_2 & 1-2x_2^2
\end{array}\right)$.
Figure~\ref{fig:hessian2d_rand} shows the pointwise error of the Hessian approximation in the point $x=(\frac{1}{2},\frac{1}{3})$ for $h$--versions of BEM with polynomial degree $p=0,1,2,3$, as well as for a $p$--version with $h=0.2$.
Figure~\ref{fig:errorbem2d} displays the corresponding BEM-error $\left\|\mu-\mu_h\right\|_V$ between $\mu_h$ and an extrapolated density $\mu$. All versions show their characteristic rate of convergence, i.e.~$1.5$, $2.5$, $3.5$, $4.5$ for the $h$-versions and exponential for the $p$-version until the error is about $10^{-8}$ at which point the quadrature errors for the outer integration in the semi-analytic evaluation of (\ref{eq:discreteFormulation}) dominate the BEM-error with
analytic computation of the involved integrals.
\end{example}
\begin{figure}[h!]
	\centering
		\includegraphics[width=120mm, keepaspectratio]{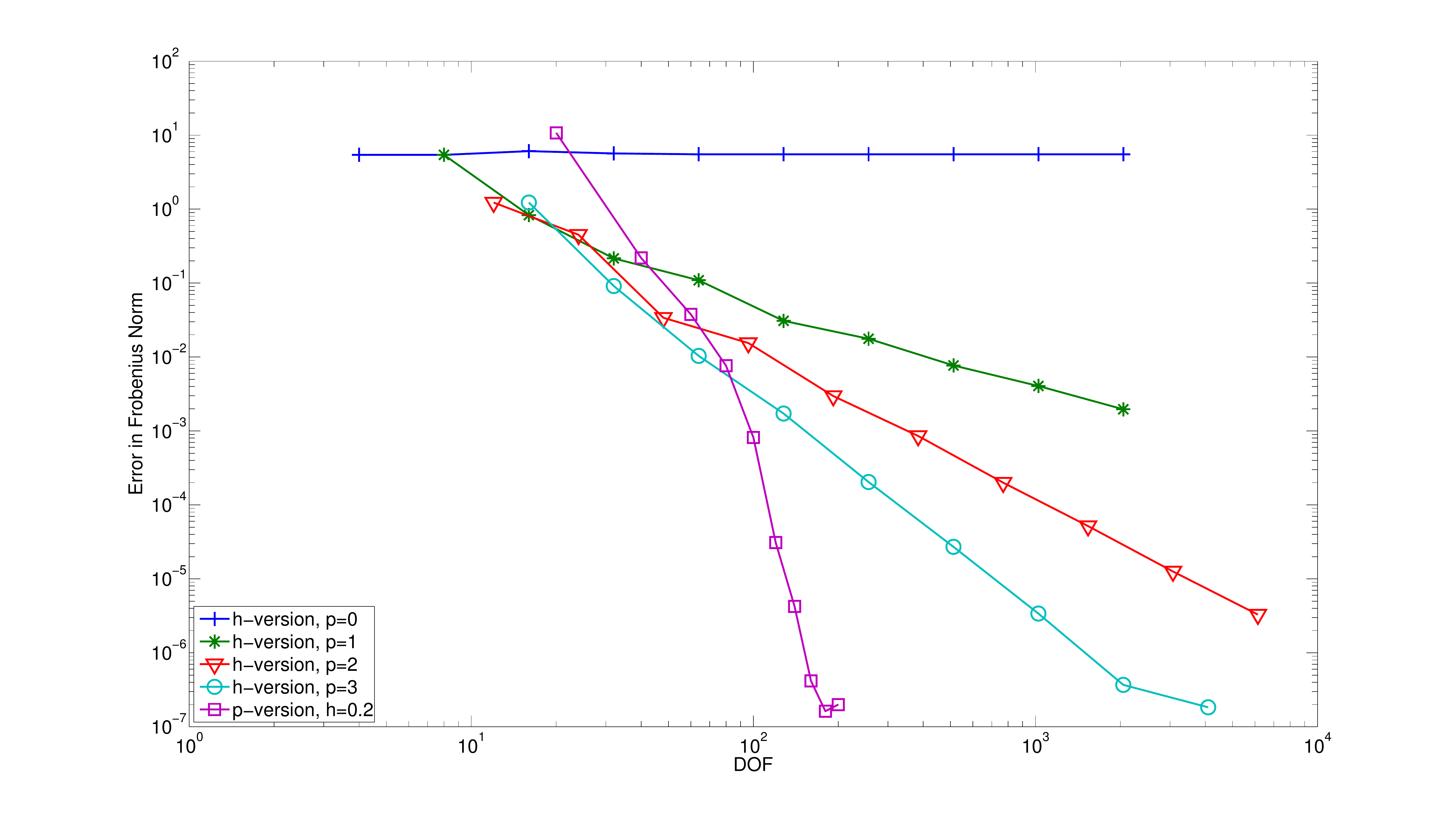}
		\caption{Error of the Hessian approximation for a point on $\Gamma_h$}
	\label{fig:hessian2d_rand}
\end{figure}

\begin{figure}[h!]
	\centering
	\includegraphics[keepaspectratio,width=120mm]{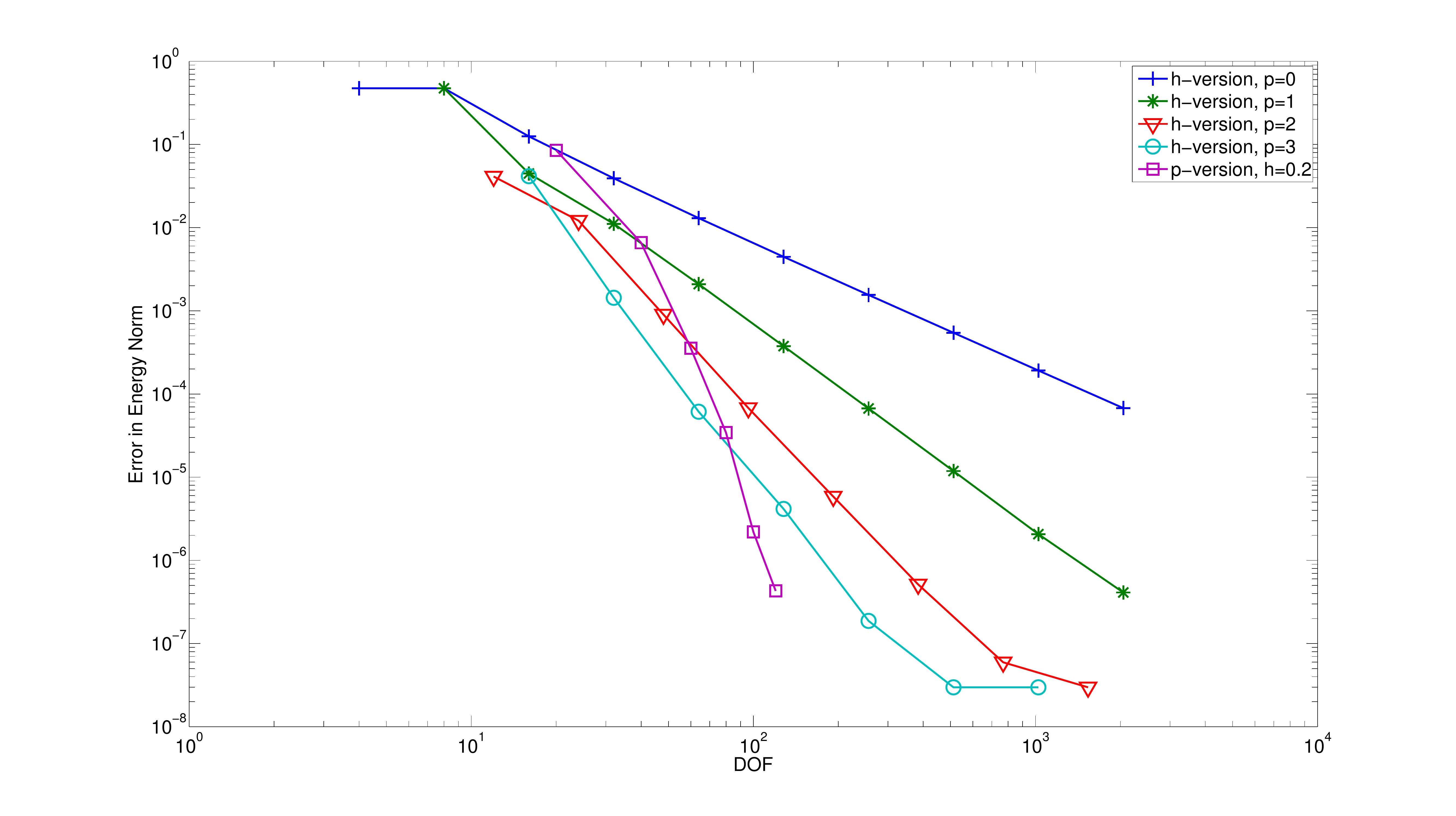}
	\caption{BEM-Error $\|\mu -\mu_h\|_V$ in the energy norm for the 2d case}
	\label{fig:errorbem2d}
\end{figure}
\begin{example}
Let $\Omega=[-1,1]^3$ be the domain and $g$ corresponding to the exact solution $u(x)=\frac{1}{\left\|x\right\|}$ with Hessian $H(x)= \frac{3}{\left\|x\right\|^5} \left(
\begin{array}[pos]{c c c}
	x_1^2 & x_1x_2 & x_1x_3 \\ x_1x_2 & x_2^2 & x_2x_3 \\ x_1x_3 & x_2x_3 & x_3^2
\end{array}\right)-\frac{1}{\left\|x\right\|^3} I$.
Figure~\ref{fig:errorbem3d} displays the BEM-error $\left\|\mu-\mu_h\right\|_V$ for three $h$-versions $(p=0,1,2)$. Figure~\ref{fig:hessian3d_outside}
shows the error in $x=(1,\frac{1}{3},\frac{1}{3})$. Again, at least for $p \geq 2$ we observe good convergence of the Hessian approximation.
\end{example}
\begin{figure}
	\centering
\includegraphics[keepaspectratio,width=120mm]{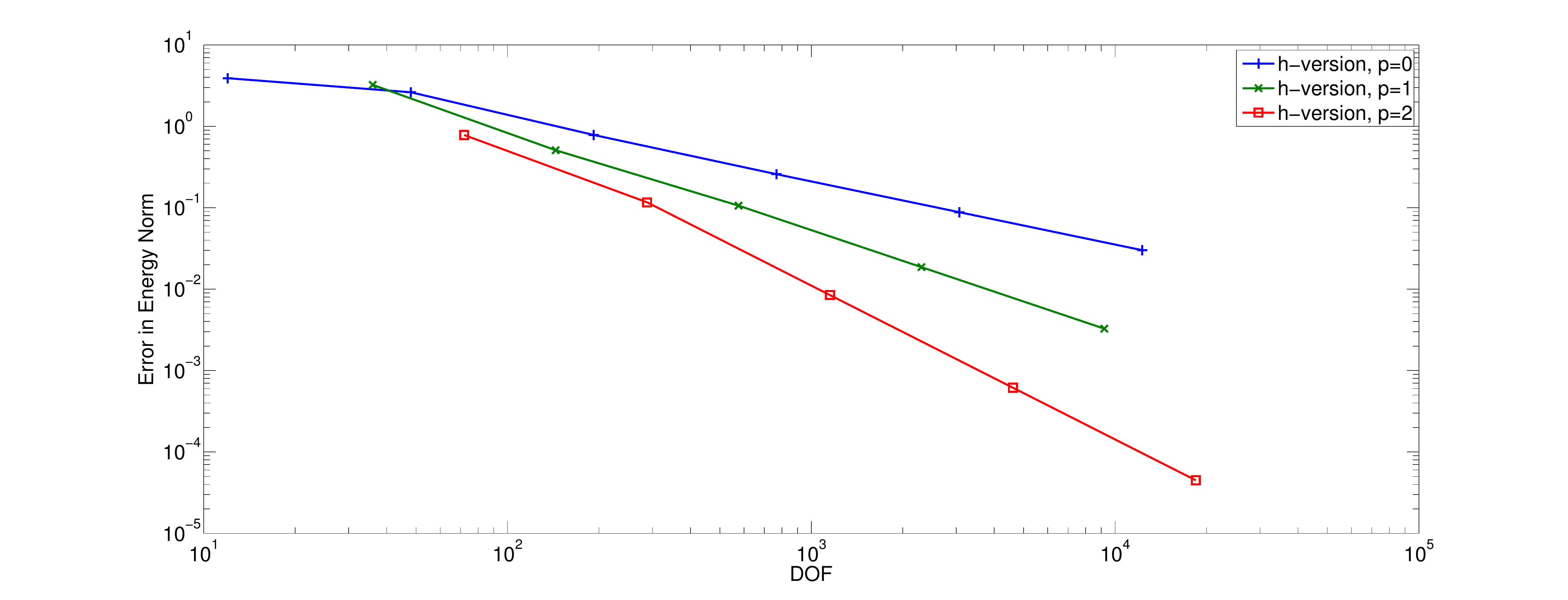}
	\caption{BEM-Error $\|\mu -\mu_h\|_V$  in the energy norm for the 3d case on the cube}
	\label{fig:errorbem3d}
\end{figure}

\begin{figure}[h!]
	\centering
		\includegraphics[width=120mm, keepaspectratio]{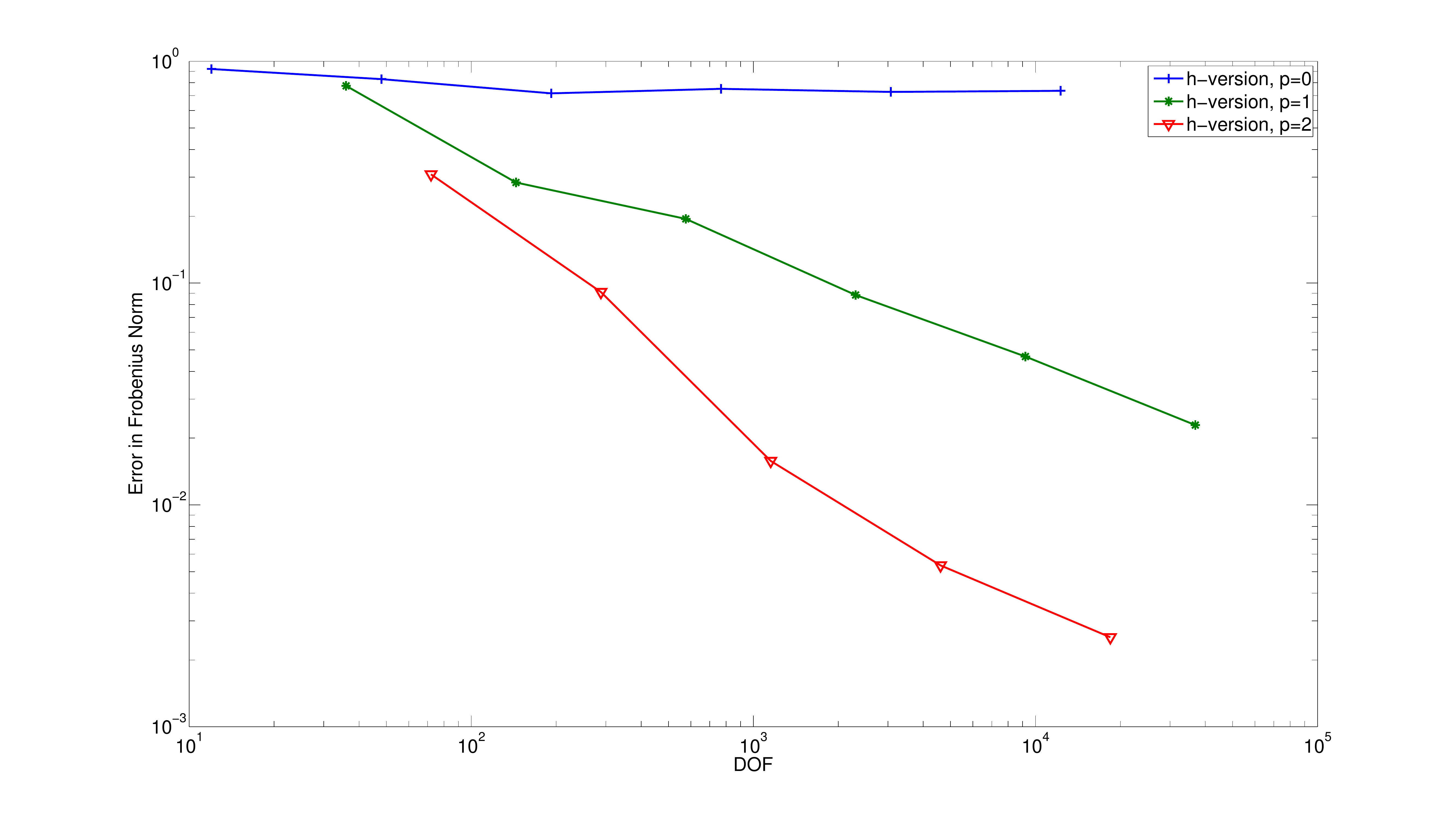}		
		\caption{Error of the Hessian approximation 3d case, $h$-versions for a point on $\Gamma$}
	\label{fig:hessian3d_outside}
\end{figure}

Similar qualitative results hold for the case of a sphere. However, in this case the error for the domain approximation quickly dominates the numerical error of differentiation \cite{costea2012}.\\

To control the error of the domain approximation, it is well known \cite{Ned00} that the finite elements to approximate the surface should be one order higher than the finite elements to approximate the solution of the integral equation with the single layer potential. As the accurate computation of the Hessian requires at least quadratic elements, third order elements should be used for $\varphi$. The corresponding software development will require further research.

\subsection{Molodensky problem}
In this section we perform a first study of the numerical solution of the fully nonlinear Molodensky problem. To assess the accuracy and convergence properties of the proposed algorithm, we study it in a simple, but algorithmically nontrivial model problem with an explicit exact solution.\\

The data $W, G$ are those of a spherical earth of radius $1.1$~with the Newton gravitational potential: For $x \in S^2$, the surface is described by $\varphi(x)=1.1 x$,
the gravitational potential is $W(x)=\frac{1}{1.1}$, and $G(x)=-\frac{x}{1.1^2}$. We choose the unit sphere $S^2$, $\varphi_0(x)=x$, as initial approximation. Therefore, $W_0(x)=1$, $G_0(x)=-x$ and $h_0=\frac{x}{2}$ ($x \in S^2$).

Since our model problem is rotationally invariant and the algorithm preserves this symmetry, we expect a sequence of computed surfaces which are slightly perturbed spheres converging to a sphere of radius 1.1.
The perturbation should be due to discretization and rounding errors. The mean $L^2$ error in $\varphi$ is thus equivalent to
\begin{equation*}
\|\varphi_m - \varphi\|_{L^2} \asymp \frac{1}{\# \ nodes}\big[ \sum_{i=1}^{ \#\ nodes} dist(\varphi_m(x_i), \varphi(S^2))^2]^{1/2} \ .
\end{equation*}

We have performed several numerical experiments with different parameters $\theta_0, \kappa$. Firstly, if the amount of data smoothing is too small, the algorithm is unstable as expected (and observed) in the case without smoothing operator.
Secondly, if the amount of data smoothing is too large, then essential information on the right hand side in the linearized Molodensky problem is lost in the first steps, and in combination with the numerical errors convergence is lost. Also if the amount of smoothing does not decay sufficiently
fast, the right hand side in the linearized Molodensky problem is close to machine precision, leading to
numerical artefacts. Figure \ref{fig:radiusl2smoother} presents some reasonable choices of parameters $\theta_0, \kappa$.

Figure \ref{fig:radiusl2smoother} shows that the propagation of the discretization error cannot be eliminated. However, increasing the amount of smoothing per iteration for a fixed mesh delays the point at which the propagated discretization error becomes dominating. Decreasing the mesh size, leads to a more even error reduction per iteration step. However, the error reduction per iteration step also decreases. The point at which discretisation errors become dominant seems difficult to predict.

Figure \ref{fig:pointwiserror} shows the average pointwise error $$|u_N(x) -u(x)|_{av} = \frac{1}{M}\left(\sum_{i=1}^{M} |u_N(x_i^{ext}) -u(x_i^{ext})|^2\right)^{1/2}$$ computed in a set of $M=10242$ exterior points for the linearized Molodensky problem with smoother ($\theta_0=2.6, \kappa=6$) for the
first three Nash-H\"ormander iteration steps. Here $u(x)$ is obtained by extrapolation. All three curves  show similar convergence rates for DOF $\to \infty$, see Table \ref{tab:errorspointwise}.

Figure \ref{fig:radiusrestart} displays the $L^2$ error in $\varphi$ versus the number of restarts for the restarted algorithm. The restart is done after each iteration step. We observe the same structural behaviour as for the other two experiments. Therefore, from the third restart of the algorithm onwards the  discretization error
propagation becomes dominating. However, refining the mesh, from $N=2$ to $N=3$, slightly reduces the error after the second and third restart, before increasing after the third restart.

Figure  \ref{fig:error_g_restart} displays the $L^2$ error in the gravity vector,
\begin{equation*}
\|G_m - G\|_{L^2} \asymp \frac{1}{\# \ nodes}\big[ \sum_{i=1}^{ \#\ nodes} (G_m(x_i)-G(x_i))^2]^{1/2} \ ,
\end{equation*}
in the algorithm with restart. The errors slowly decreases for the first five iteration steps, but from this step onwards the method provides uncontrollable surface updates (peaks and undesirable deformations occur) afterwards.

To sum up, we observe convergence of the algorithm in $\varphi$ until numerical errors due to discretization and domain approximation accumulate. A higher-order discretization of the surface seems to be necessary to obtain stable convergence and an improved approximation of $G$ after a larger number of steps. For the restarted algorithm, a better understanding of the optimal smoothing parameters must be achieved.
\begin{figure}[tbp]
	\centering
	\includegraphics[keepaspectratio,width=120mm]{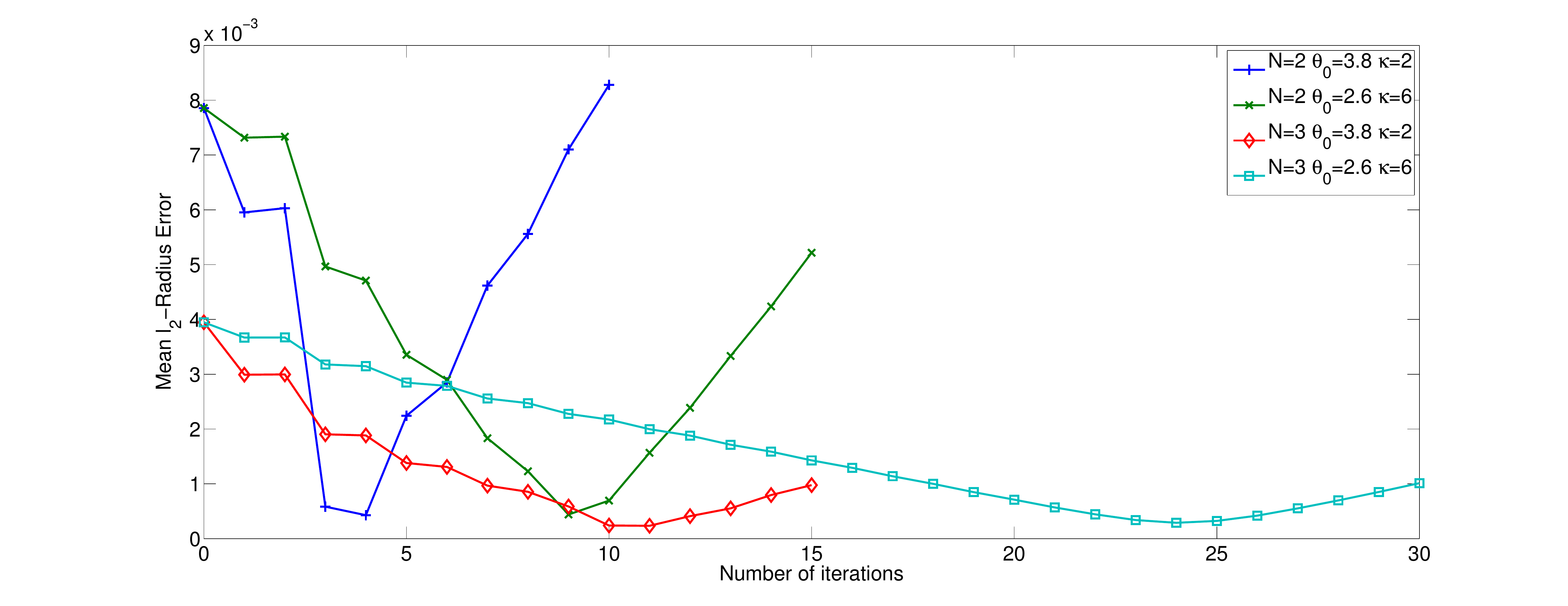}
	\caption{$\|\varphi_m - \varphi\|_{L^2}$ for the algorithm without restart}
	\label{fig:radiusl2smoother}
\end{figure}
\begin{figure}[tbp]
	\centering
	\includegraphics[keepaspectratio,width=120mm]{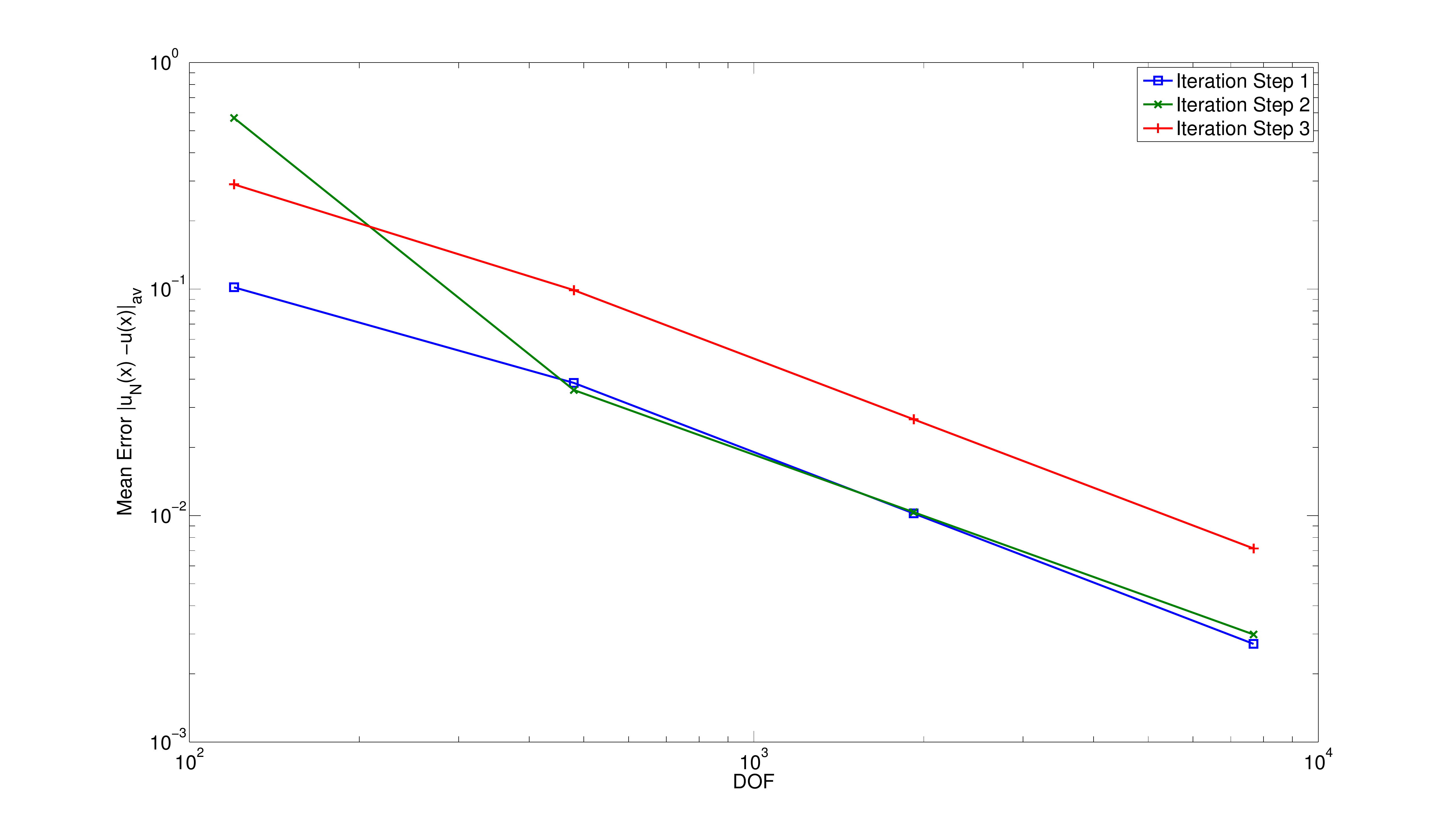}
	\caption{Pointwise error $|u_N(x) -u(x)|_{av}$ computed in a set of $10242$ exterior points for the linearized Molodensky problem}
	\label{fig:pointwiserror}
\end{figure}

\begin{figure}[h!]
	\centering
	\includegraphics[keepaspectratio,width=120mm]{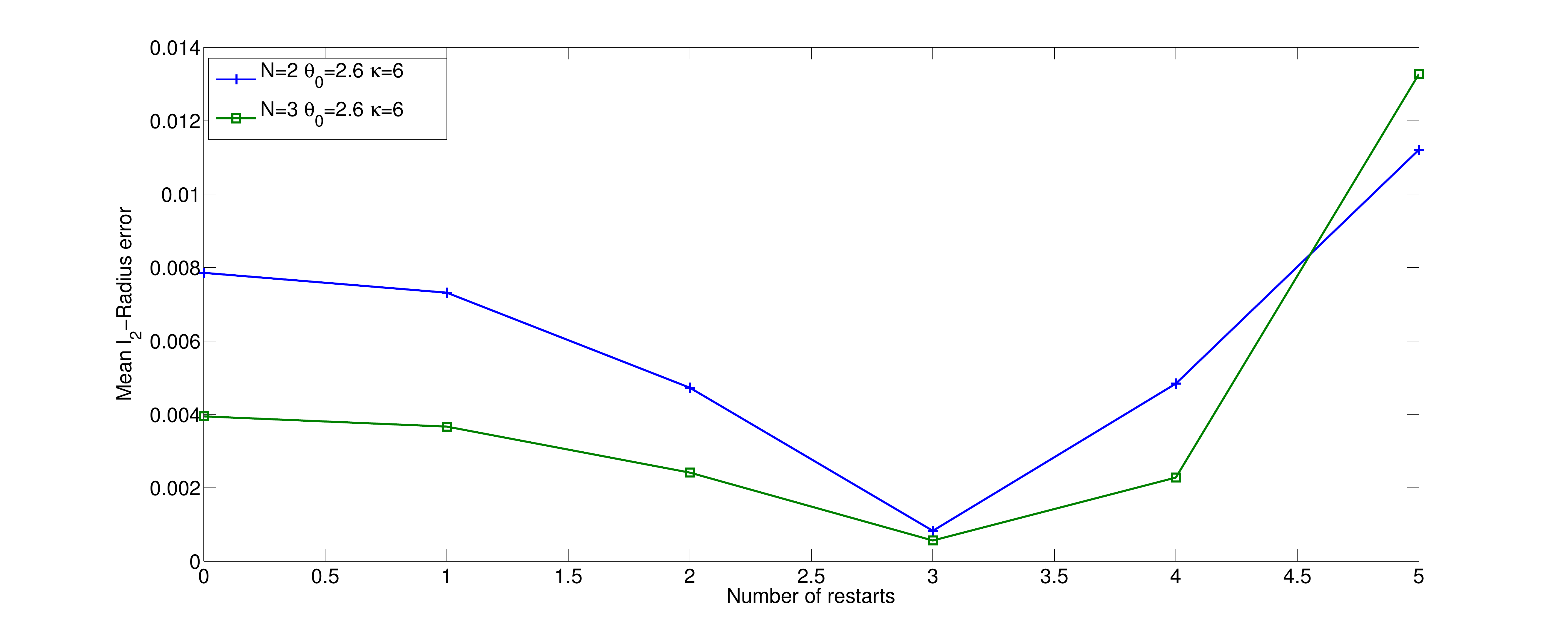}
	\caption{$\|\varphi_m - \varphi\|_{L^2}$ for the algorithm with restart}
	\label{fig:radiusrestart}
\end{figure}

\begin{figure}[h!]
	\centering
	\includegraphics[keepaspectratio,width=120mm]{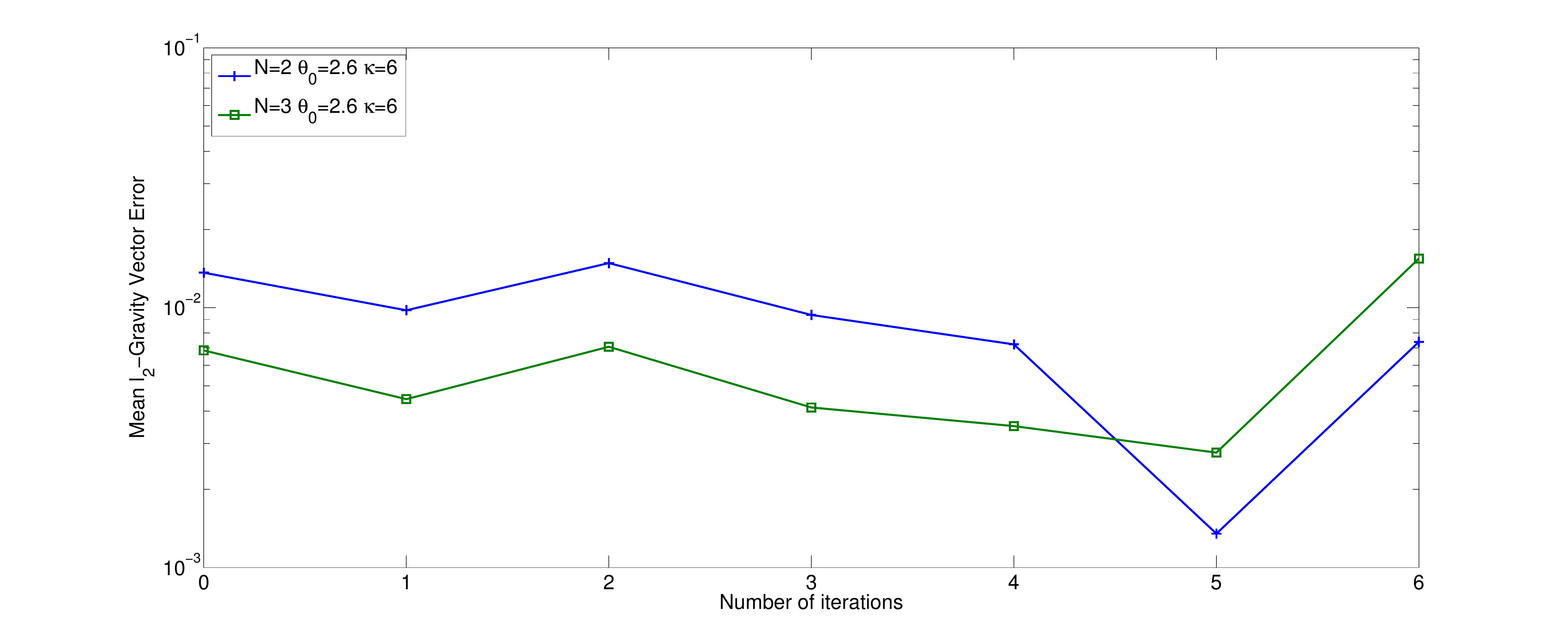}
	\caption{$\|G_m - G\|_{L^2}$ for the algorithm with restart}
	\label{fig:error_g_restart}
\end{figure}
\begin{table}[h!]
\begin{center}
\begin{tabular}{|r|r|l|l|l|}
\hline
Iter  & DOF & $|u_N(x)-u(x)|_{av}$ & EOC \\
\hline
0&  120    &\hspace{0.7cm}0.10170&  \\
  &  480    &\hspace{0.7cm}0.03850& 0.70  \\
   &1920    & \hspace{0.7cm}0.01022& 0.96 \\
   &7680    & \hspace{0.7cm}0.00271& 0.96 \\
  \hline
1&  120    &\hspace{0.7cm}0.56875	&  \\
   &  480    &\hspace{0.7cm}0.03582& 1.99  \\
   &1920    & \hspace{0.7cm}0.01034&  0.90\\
   &7680    & \hspace{0.7cm}0.00299& 0.90\\
\hline
2&  120    &\hspace{0.7cm}0.28961&  \\
 &  480    &\hspace{0.7cm}0.09885& 0.77 \\
  &1920    & \hspace{0.7cm}0.02660&  0.95\\
  &7680    & \hspace{0.7cm}0.00716& 0.95\\
\hline
\end{tabular}
\caption{Pointwise errors and experimental orders of convergence for the linearized Molodensky problem}
\label{tab:errorspointwise}
\end{center}
\end{table}

Figure \ref{fig:figuressm2}  displays the sequence of obtained surfaces. The marked point is always the north pole of the sphere, i.e.~$x=y=0$ and only $z$ varies.  Interestingly, for each experiment  the surface update is almost constant over the mesh points, leading to a sequence of almost spheres.

\begin{figure}[h!]
	\centering
	 \subfigure[320 triangles, $z=1.007$]{
		\includegraphics[width=69.5mm, keepaspectratio]{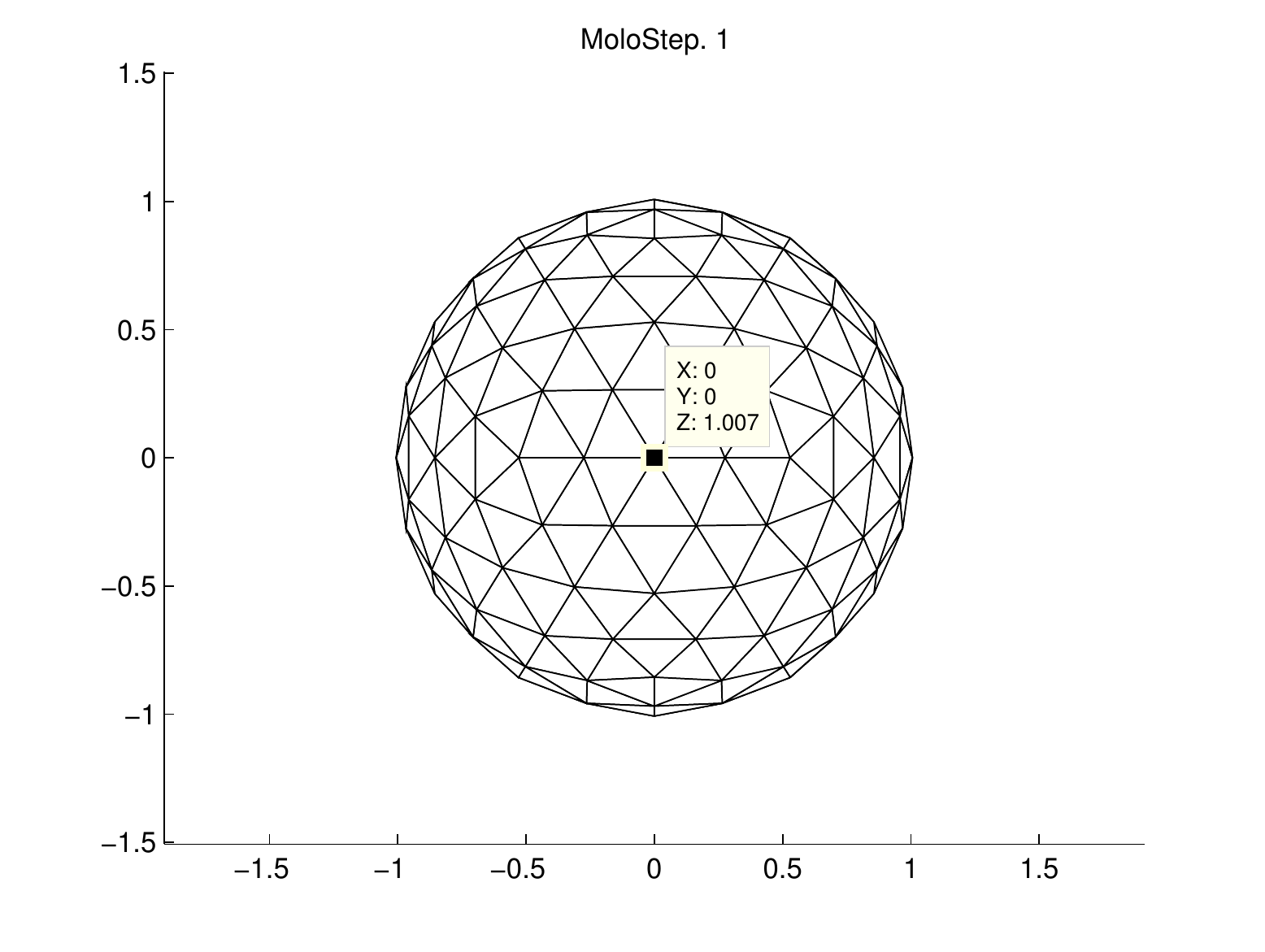}} \
	 \subfigure[1280 triangles, $z=1.007$]{
		\includegraphics[width=69.5mm, keepaspectratio]{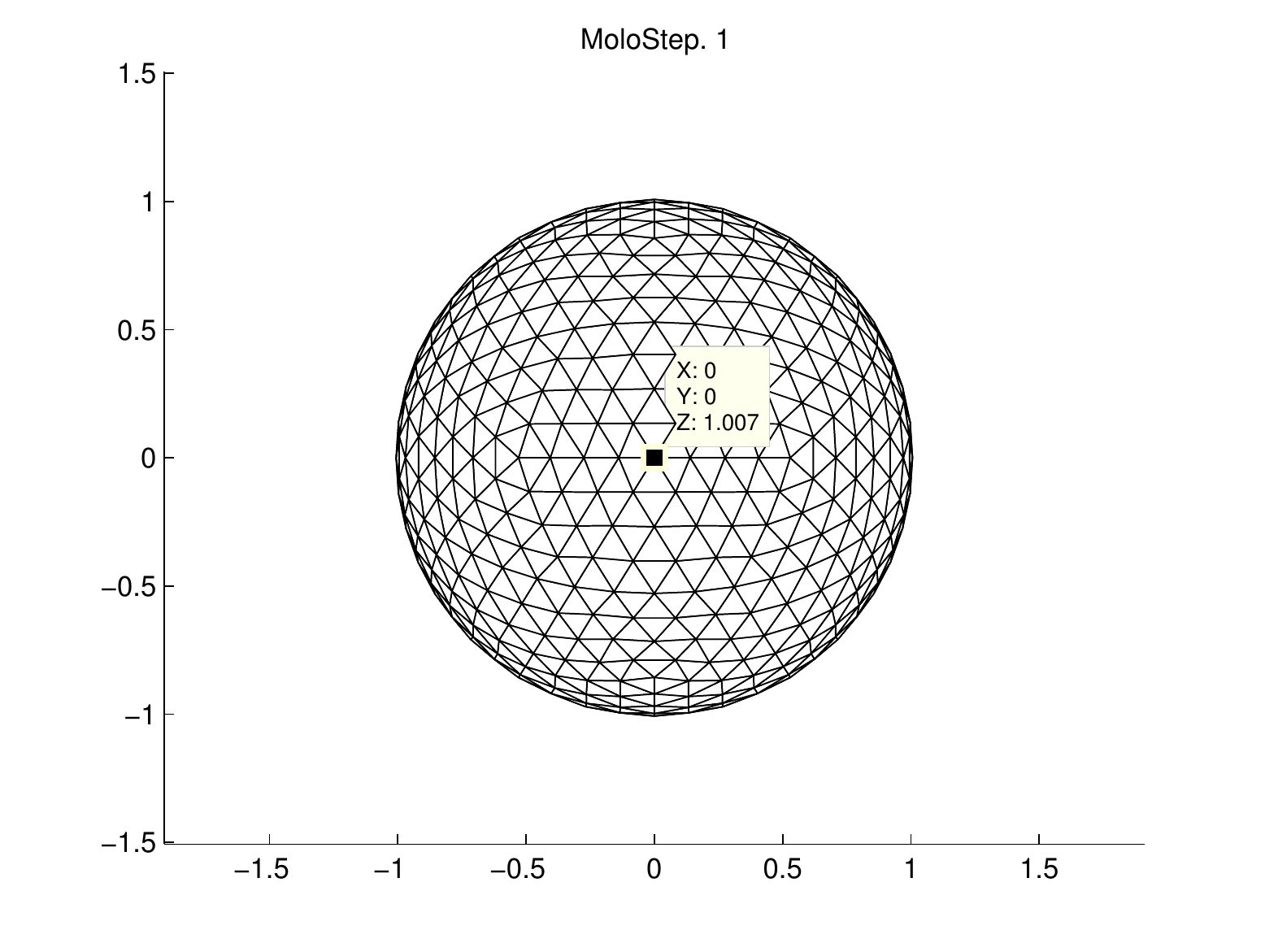}}\\
	 \subfigure[320 triangles, $z=1.057$]{
		\includegraphics[width=69.5mm, keepaspectratio]{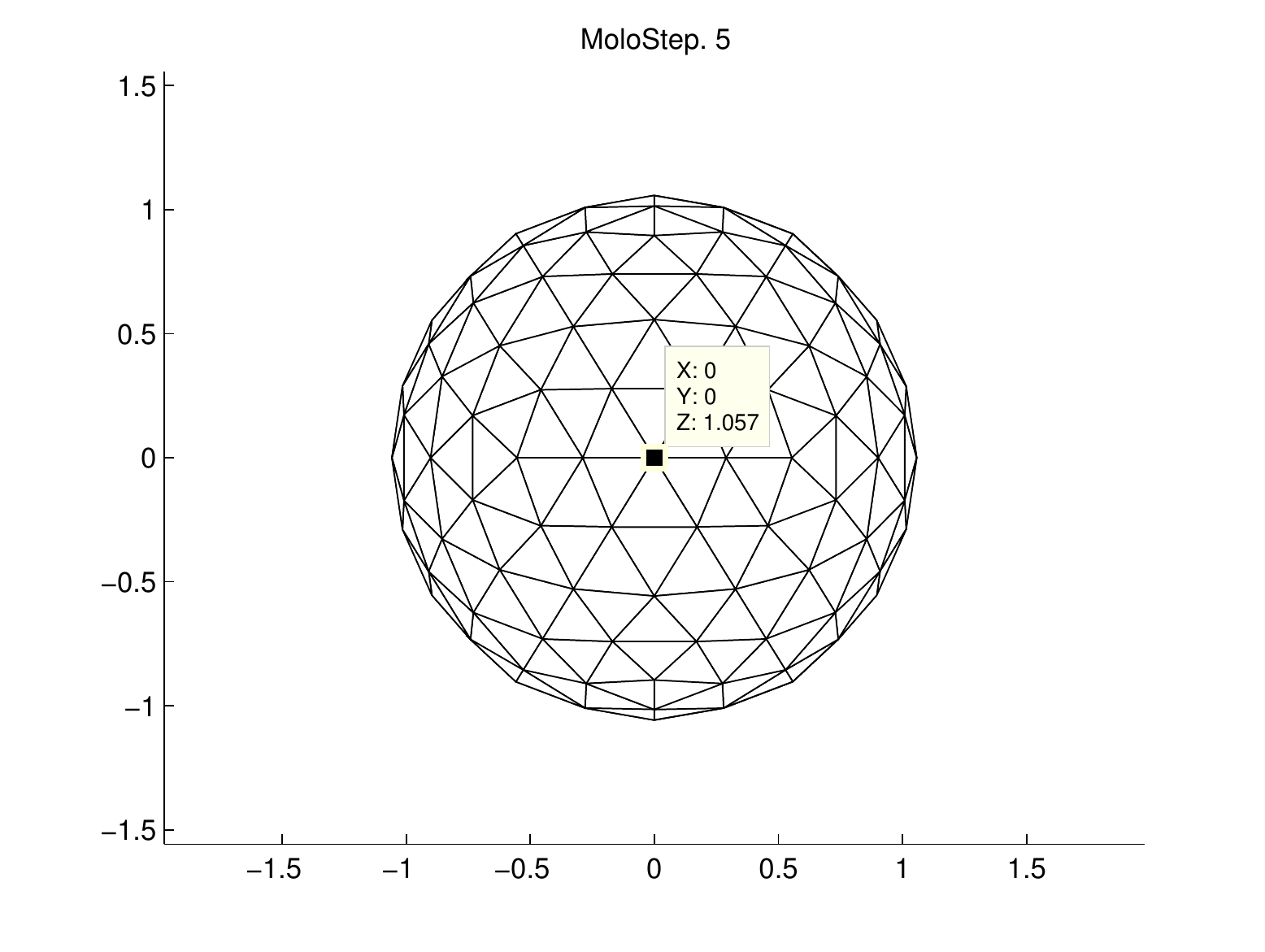}} \
	 \subfigure[1280 triangles, $z=1.045$]{
		\includegraphics[width=69.5mm, keepaspectratio]{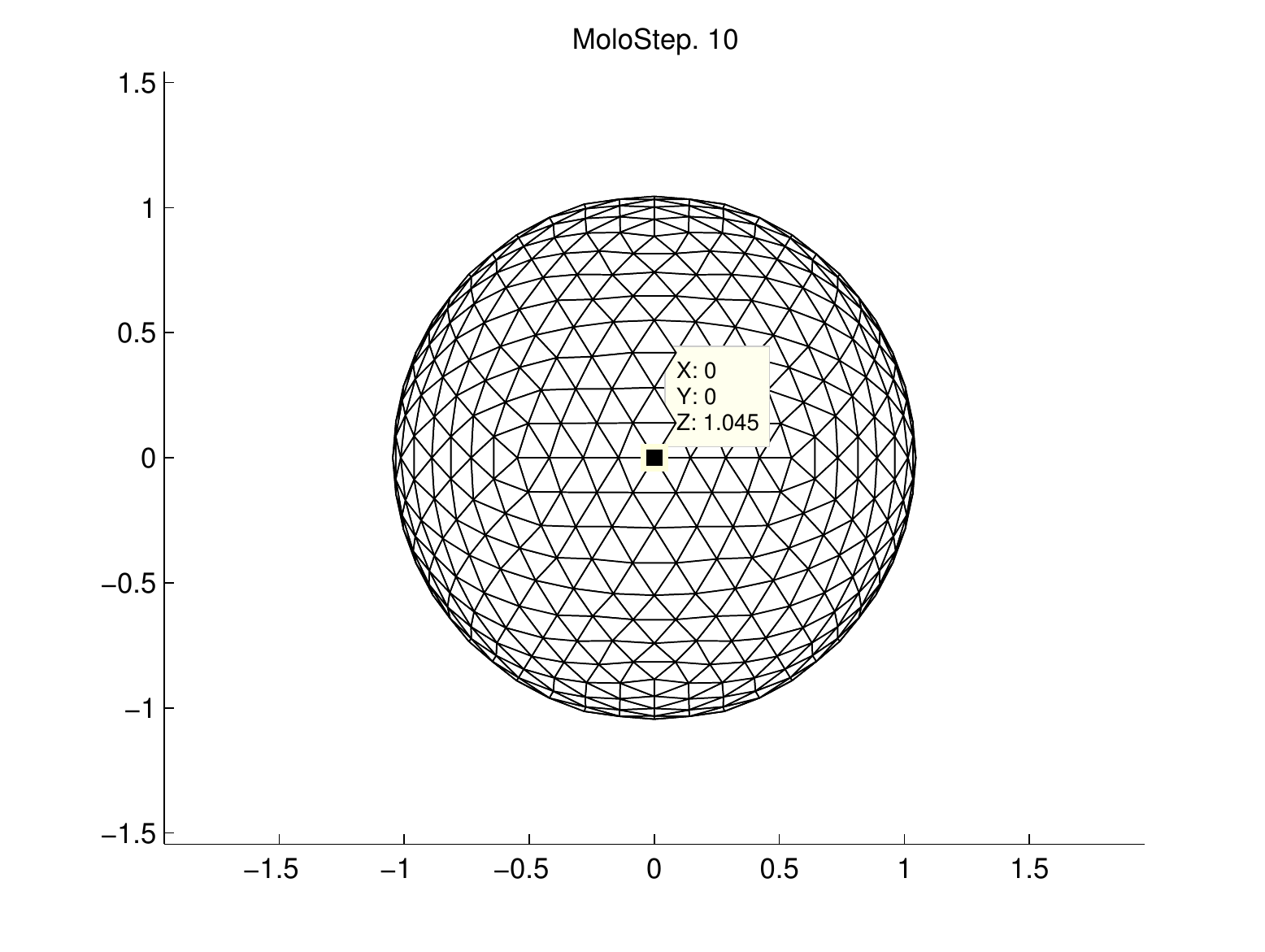}}\\
\subfigure[320 triangles, $z=1.107$]{
		\includegraphics[width=69.5mm, keepaspectratio]{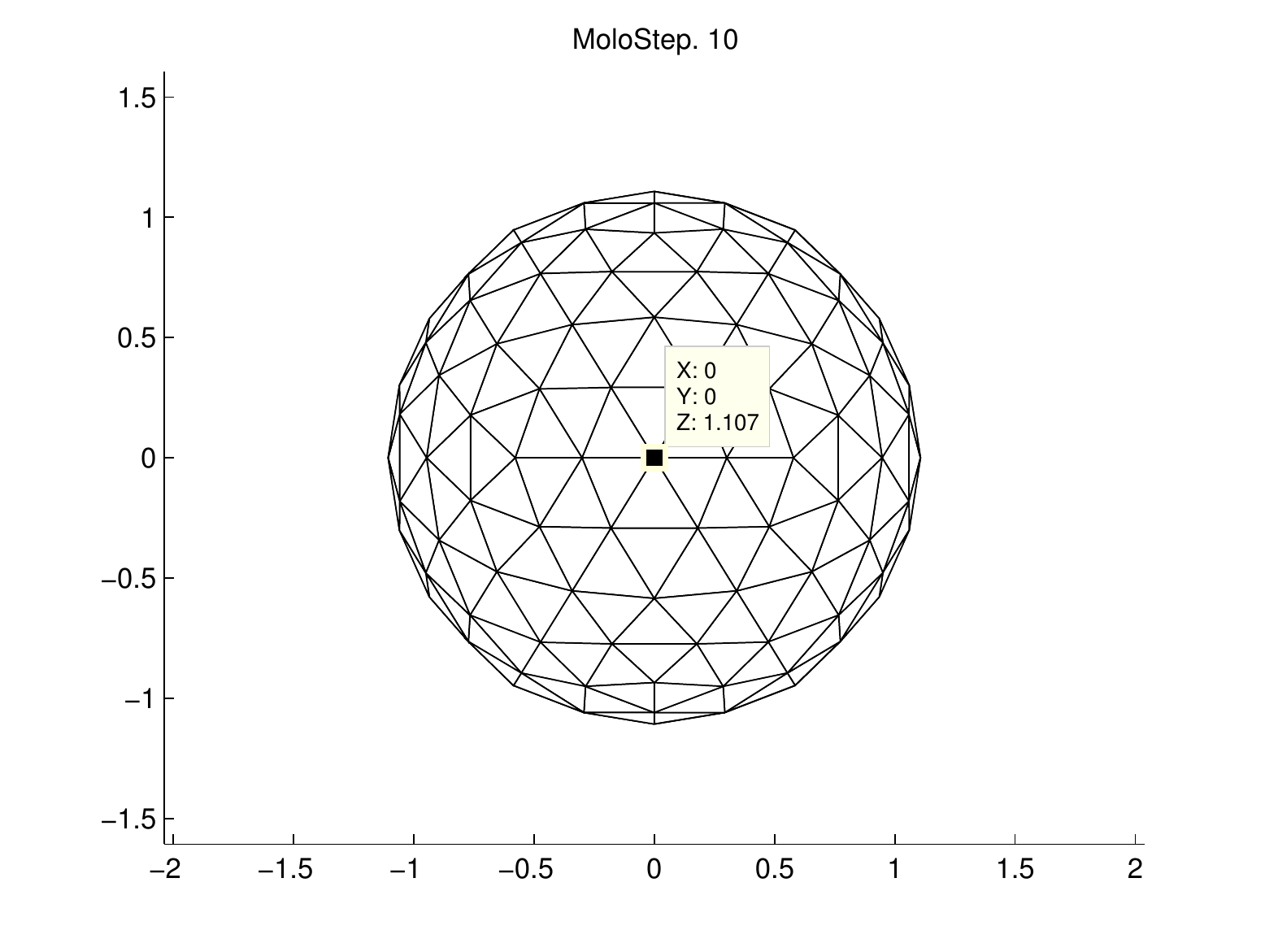}} \
	 \subfigure[1280 triangles, $z=1.104$]{
		\includegraphics[width=69.5mm, keepaspectratio]{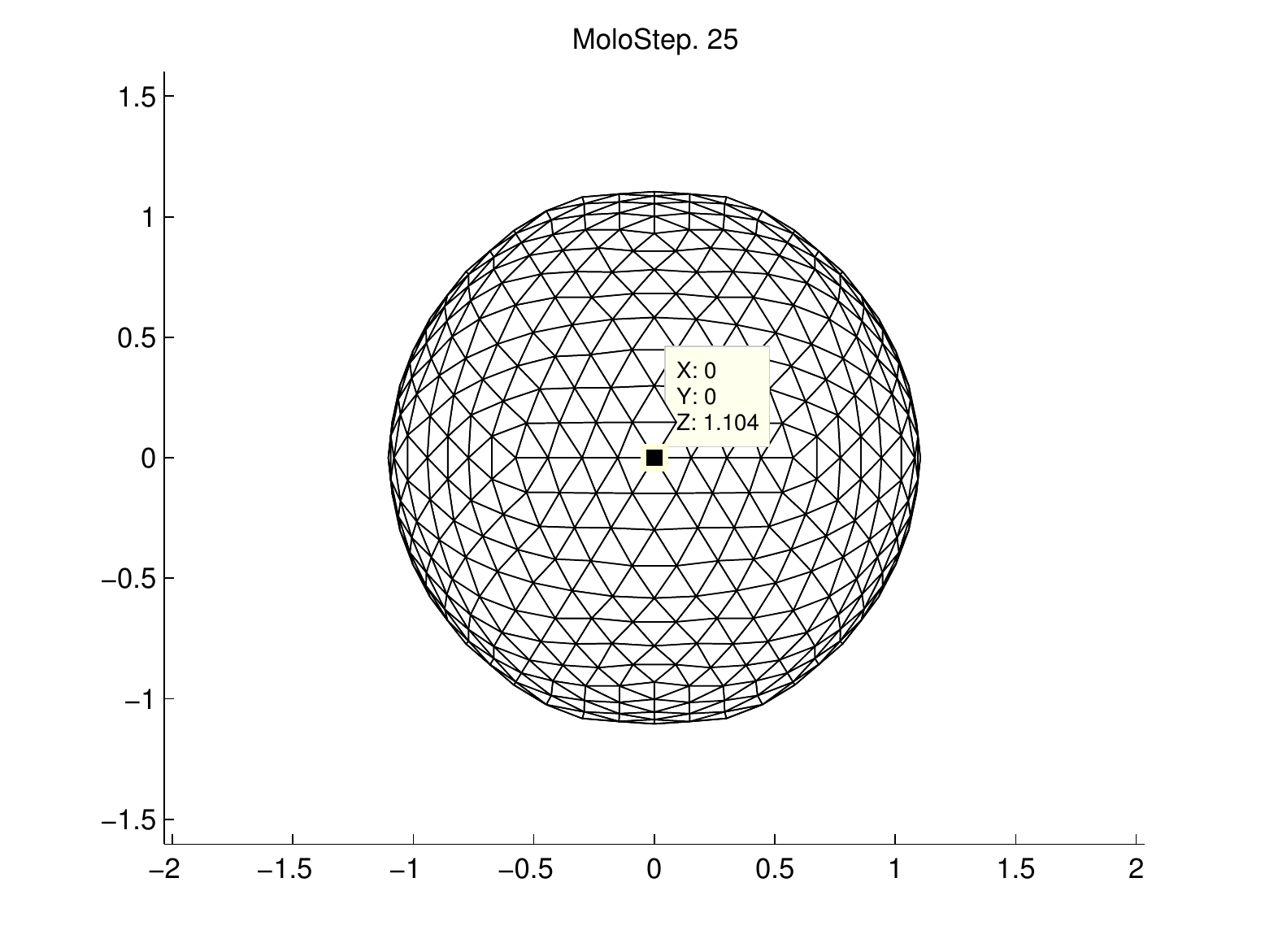}}		
		\caption{Icosahedron refinements for $\theta_0=2.6, \kappa=6$}
	\label{fig:figuressm2}
\end{figure}

\section{Appendix}
We are going to use the following definition of the H\"{o}lder--spaces, but will rarely use them for integer exponent $a$:
\begin{definition}[Definition A.3~\cite{Hoermander}]
Let $k \in \mathbb N_0, k<a \leq k+1$ and $B \subseteq \mathbb R^n$ compact, convex such that $\mathring{B} \neq \emptyset$.\\
Define
\begin{align*}
\mathscr H^{a}(B):= \lbrace u \in C^k(B): &\|u\|_0=\sup \limits_{x \in B} |u(x)|<\infty \quad \mbox{and} \\
& |u|_a:=\sum_{|\alpha|=k} \sup \limits_{x \neq y \in B} \frac{|\partial^\alpha u(x)-\partial^{\alpha}u(y)|}{|x-y|^{a-k}}<\infty \rbrace.
\end{align*}
We also set $\mathscr H^0(B):=C(B)$. Then $\mathscr H^a := \mathscr H^a(B)$ with the norm $\|\cdot\|_a:=\|\cdot\|_0+|\cdot|_a$ is a Banach space.
\end{definition}

For a compact manifold, one defines $\mathscr H^a$ by covering it with a finite number of neighborhoods homeomorphic to subsets of $\mathbb{R}^n$.\\

Basic interpolation estimates will be used frequently:
\begin{equation}\label{interpol}
\|v\|_{\sigma a + (1-\sigma)b} \leq C \|v\|_a^{\sigma} \|v\|_{b}^{1-\sigma} \qquad (\sigma \in (0,1), \ v \in \mathscr H^{\max\{a, b\}}) \ .
\end{equation}
Similarly, if $(a,b) \in \mathbb{R}_+^2$ belongs to the convex hull of $(a_1, b_1), \dots, (a_J,b_J)\in \mathbb{R}_+^2$, then
\begin{equation}\label{interpol2}
\|v\|_{a}\|w\|_{b} \leq C \sum_{j=1}^{J}\|v\|_{a_j} \|w\|_{b_j} \qquad (\ v \in \mathscr H^{\max\{a_j\}}, \ w \in \mathscr H^{\max\{b_j\}}) \ .
\end{equation}

\subsection{Proof of Theorem \ref{apriorithm}}\label{app:proofs}

For ease of presentation, we set $u = (u^{(1)},u^{(2)}):=(W, \varphi) : S^2 \to \mathbb{R} \times \mathbb{R}^3$, $f := (G,W) : S^2 \to  \mathbb{R}^3 \times \mathbb{R}$. The map from $(W,\varphi)$ to the corresponding $G$ is denoted by $\Gamma$. The Molodensky problem assumes the form $\Phi(u) := (\Gamma(u), u^{(1)}) = f$.\\

In this notation, the Nash--H\"{o}rmander iteration reads as
$$u_{m+1} = u_m + \Delta_m \dot{u}_m, \quad \dot{u}_m = \Psi(v_m) g_m, \quad v_m = S_{\theta_m} u_m \ ,$$
$\Psi$ being the inverse operator to the linearization of $\Phi$ around $v_m$.
To show that the Algorithm \ref{alg:nashhormsm} is a reformulation of the one in \cite{Hoermander}, we show that the equations (\ref{wreccurencetosmooth}) and (\ref{greccurencetosmooth}) are equivalent to usual definition
\begin{equation}\label{gdefinition}
\Delta_0 g_0 = S_{\theta_0} f, \quad g_m = \Delta_m^{-1} ((S_{\theta_m}-S_{\theta_{m-1}})(f-E_{m-1})-\Delta_{m-1}S_{\theta_m}e_{m-1})\quad (m>0) \ .
\end{equation}
Here, the errors $e_m= e_m' + e_m''$ are defined as the sum of a smoothing error
\begin{equation}\label{epdef}
e_m' = (\Phi'(u_m)-\Phi'(v_m))\dot{u}_m
\end{equation}
and a linearization error
\begin{equation}\label{eppdef}
e_m'' = \Delta_m^{-1}(\Phi(u_{m}+\Delta_m \dot{u}_m)-\Phi(u_m)-\Delta_m\Phi'(u_m)\dot{u}_m)\ .
\end{equation}
The total error is $E_m = e_0 + \cdots + e_{m-1}$ ($e_0 = E_0 = 0$). \\

In our case, $\Phi'(u)\dot{u} = (\Gamma'(u)\dot{u}, \dot{u}^{(1)})$, so that with $\dot{G}_m = \Gamma'(W_m, \varphi_m)(\dot{W}_m, \dot{\varphi}_m)$
\begin{eqnarray*}
e_m' &=& \Phi'(W_m, \varphi_m)(\dot{W}_m, \dot{\varphi}_m)-\Phi'(\widetilde{W}_m, \widetilde{\varphi}_m)(\dot{W}_m, \dot{\varphi}_m)\\
& =& (\Gamma'(W_m, \varphi_m)(\dot{W}_m, \dot{\varphi}_m), \dot{W}_m) - (\Gamma'(\widetilde{W}_m, \widetilde{\varphi}_m)(\dot{W}_m, \dot{\varphi}_m), \dot{W}_m)\\
& =& (\dot{G}_m - g_m^{(1)},0) \ .
\end{eqnarray*}
Similarly, for the linearization error we have
$$\Delta_m e_m'' = \Phi(W_m+\Delta_m\dot{W}_m, \varphi_m+ \Delta_m\dot{\varphi}_m)-\Phi(W_m, \varphi_m)-\Delta_m\Phi'(W_m, \varphi_m)(\dot{W}_m, \dot{\varphi}_m) \ ,$$
with second component $W_m+\Delta_m\dot{W}_m - W_m - \Delta_m\dot{W}_m = 0$. The first component, by definition, is $G_{m+1}-G_m-\Delta_m \dot{G}_m$.

As a consequence, the second components of the errors $e_m$ and $E_m$ vanish. Equation (\ref{gdefinition}) yields
\begin{equation}\label{g2alg}\Delta_0 g_0^{(2)} = S_{\theta_0}(W-W_0)\ ,\quad \Delta_m g_m^{(2)} = S_{\theta_m}(W-W_0)-S_{\theta_{m-1}} (W-W_0) \quad (m>0) .\end{equation}
Concerning the first component of $g_m$, we use $\Delta_0 g_0^{(1)} =  S_{\theta_0}(G-G_0)$ and our above computation:
\begin{eqnarray*}
\Delta_1 g_1^{(1)} &=& (S_{\theta_1}-S_{\theta_{0}})(G-G_0)-\Delta_0 S_{\theta_1}(\dot{G}_0 - g_0^{(1)})-S_{\theta_1}(G_1-G_0-\Delta_0\dot{G}_0) \\
&=& S_{\theta_1}(G-G_1+\Delta_0g_0^{(1)})-S_{\theta_0}(G-G_0) \ .
\end{eqnarray*}
Similarly we obtain a recursion formula for $m>1$:
\begin{equation}\label{g1alg}\Delta_m g_m^{(1)} = S_{\theta_m}\left(G-G_m+\sum_{j=0}^{m-1}\Delta_jg_j^{(1)}\right)-S_{\theta_{m-1}}\left(G-G_{m-1}+\sum_{j=0}^{m-2}\Delta_jg_j^{(1)}\right) \ .\end{equation}
Equations (\ref{g2alg}) and (\ref{g1alg}) show the equivalence of Algorithm \ref{alg:nashhormsm} to the formulation in \cite{Hoermander}.

\begin{proof}[of Theorem \ref{apriorithm}]By the above reduction, it is sufficient to analyze H\"ormander's iterative method and derive an a priori estimate for the error after $m$ steps. Our proof relies on certain estimates in H\"{o}rmander's qualitative analysis. Recall that the smoothing operator is generated by $(-1)^k \Delta^k$.\\

The proof is given in several steps. We are going to rely on a number of auxiliary results.

Given a sufficently large, fixed $a_\Phi$, we recall the following continuity estimates for $\Phi''$ and the inverse of the linearization $\Psi$ from (\cite{Hoermander}, equations (2.1.5/6)):\\

1.) For all $\epsilon>0$, $0\leq a \leq a_\Phi$ and $u, v,w\in C^\infty$ with $\|u\|_{2+\epsilon}<C$  we have:
\begin{equation}  \label{Phi215}
\|\Phi''(u;v,w)\|_{a+2\epsilon} \lesssim \|v\|_{a+2+3\epsilon}\|w\|_{0}+\|v\|_{0}\|w\|_{a+2+3\epsilon} + \|v\|_0\|w\|_0\|u\|_{a+3+2\epsilon}
\end{equation}

2.) For all $\epsilon>0$, $0\leq a \leq a_\Phi$ and $v,g\in C^\infty$ with $\|v\|_{2+\epsilon}<C$
\begin{equation}  \label{psi216}
\|\Psi(v)g\|_{a+\epsilon} \lesssim \|g\|_{a+\epsilon} + \|g\|_{\epsilon} \|v\|_{a+2+\epsilon}
\end{equation}

The first lemma translates a bound for the first $m$ increments $\dot {u}_j$ into properties of $U_m= \sum_{j=0}^m \triangle_j \dot u_j$.

\begin{lemma}\label{analmol:lemma1}
Let $\epsilon >0,\, \alpha + \epsilon \notin \mathbb N, \,-\epsilon \leq \alpha_{-} \leq \alpha \leq \alpha_{+}$. Assume that $2k > \alpha+\epsilon - a$ and that for some $\delta>0$, $m \geq 0$
\begin{equation}\label{analmol:dotuj}
 \|\dot {u}_j \|_{a+\epsilon} \leq \delta \theta_j^{a-\alpha-1} \qquad  \forall 0\leq j \leq m \ \forall a \in [\alpha_-,\alpha_{+}]\ .
\end{equation}
Then $U_m= \sum_{j=0}^m \triangle_j \dot u_j \in \mathscr {H}^{\alpha+\epsilon}$ satisfies
\begin{align}
\|U_m\|_a &\leq C_1 \delta \qquad \forall a \leq \alpha+\epsilon \ ,\\
 \|U_m-S_{\theta_{m+1}}U_m\|_a &\leq C_2 \delta \theta_{m+1}^{a-\alpha-\epsilon} \qquad \forall\: 0\leq a\leq \alpha_{+} +\epsilon \ ,\\
\|S_{\theta_{m+1}}U_m\|_a &\leq {C}_3 \delta \theta_{m+1}^{(a-\alpha-\epsilon)_{+}} \qquad \forall\: 0 \leq a \leq a_0
\end{align}
for fixed $a_0$.
\end{lemma}

We refer to \cite[Lemma 2.2.1]{Hoermander}) for the proof.\\

The lemma implies that under the given assumptions all iterates $u_m$ will remain in a neighborhood of $u_0$. We may therefore localize and appeal to estimates valid near $u_0$ and justify the linearization of the problem. A quantitative formulation of the localization is as follows:
\begin{cor}
Let $\widetilde \epsilon>0$, $\mu \leq \alpha + \epsilon$, and $a$, $\delta$ as above. Define
\begin{align*}
 V_s&:= \lbrace u \in C^{\infty}: \|u-u_0\|_\mu \leq s \rbrace \ ,\\
 \widetilde{V}_s&:= \lbrace u \in C^{\infty}: \|u\|_\mu \leq s \rbrace\ .
\end{align*}
Then there exist $C, C'>0$ (which depend only on $C_1$ and the constants in Theorem \ref{heatthem}(i), (iii')) such that for all $\theta \geq C\left(\frac{\widetilde \epsilon}{\|u_0\|_a}\right)^{\frac{1}{a-\mu}}$ and all $\delta < C' \tilde \epsilon$:\\
a) $S_\theta u_0 \in V_{\widetilde \epsilon/2}$\\
b) If $a=\mu \leq \alpha+\epsilon$, then $U_k,S_\theta U_k \in \widetilde{V}_{\widetilde \epsilon/2}$, $u_{k+1}=u_0 +U_k \in V_{\widetilde \epsilon}$ and $S_{\theta_{k+1}} u_{k+1} \in V_{\widetilde \epsilon}$.
\end{cor}
We note some related estimates for the smoothed iterates $v_m = S_{\theta_m}u_m$, which in particular hold for $b = \alpha+\epsilon$:
\begin{cor}
Under the above assumptions there holds:
\begin{align}
\|u_j-v_j\|_c &\leq C (\|u_0\|_b \theta_j^{c-b} + \delta \theta_j^{c-\alpha-\epsilon}) \qquad \forall\: c\leq \alpha +\epsilon, \ 0\leq b - c<2k \label{analmol:star}\\
 \|v_j\|_c &\leq  C(\|u_0\|_b \theta_j^{c-b} + \delta \theta_j^{(c-\alpha-\epsilon)_+})\qquad \forall\: c\leq c_0, \ b\leq c \label{analmol:twostar}
\end{align}
\end{cor}
\begin{proof}
Indeed, using Property (iii') of $S_\theta$ and Lemma \ref{analmol:lemma1},
\begin{align}
 \|u_j-v_j\|_c&\leq \|u_0-S_{\theta_j}u_0\|_{c}+ \|U_{j-1}-S_{\theta_j}U_{j-1}\|_c\\
& \leq C \|u_0\|_b \theta_j^{c-b} +C \delta \theta_j^{c-\alpha-\epsilon}\nonumber.
\end{align}
Similarly,
$\|v_j\|_c \leq \|S_{\theta_j}u_0\|_c+ \|S_{\theta_j}U_{j-1}\|_c \leq C \theta_j^{c-b}\|u_0\|_b +C \delta \theta_j^{(c-\alpha-\epsilon)_{+}}$
\end{proof}

The a priori estimate of Theorem \ref{apriorithm} will be shown by induction in $m$.\\

As hypothesis we assume (\ref{analmol:dotuj}), that $\|f\|_{\alpha+\epsilon}$ is small and that the assumptions of Lemma \ref{analmol:lemma1} are verified for a suitable $\delta$. We are going to deduce the corresponding assertions with $m$ replaced by $m+1$.  \\

Together with the induction hypothesis, Lemma \ref{analmol:lemma1} and the above corollaries provide bounds of the H\"{o}lder norms of $u_j$, $v_j$ and $u_j-v_j$ for $0\leq j\leq m$. To estimate $\dot u_{m+1}$ in terms of these data, note that by definition of $\dot u_{m+1}$ and $g_{m+1}$,
\begin{align*}
\| \dot u_{m+1}\|_{a+\epsilon} &\leq C(\|g_{m+1}\|_{a+\epsilon}+\|g_{m+1}\|_{\epsilon} \theta_{m+1}^{(2+a-\alpha)_{+}})\\
g_{m+1}&= \widetilde S_m(f-E_m)-\frac{\triangle_m}{\triangle_{m+1}} S_{\theta_{m+1}}e_m
\end{align*}
where $E_m=\sum_{j=0}^{m-1} \triangle_j e_j$ and $\widetilde S_m = \frac{S_{\theta_{m+1}}-S_{\theta_{m}}}{\theta_{m+1}-\theta_{m}}$.
Writing $\widetilde{S}_m$ as an average of $\frac{d}{d\theta} S_\theta$, $\widetilde{S}_m f$ can be bounded using Property (iv) of the smoothing operator,
\begin{equation*}
\|\tilde S_m f\|_b \leq C \theta_m^{b-c-1} \|f\|_c \ ,
\end{equation*}
for any $b,c$. Properties (ii) and (iv) of the smoothing operator give similar estimates for the other terms:
\begin{align}
 \|S_{\theta_{m+1}}e_m\|_{b} &\leq C \theta_m^{(b-c')_{+}} \|e_m\|_{c'} \label {analmol:S1}\\
\|\widetilde{S}_mE_m\|_{b} &\leq  \theta_m^{b-c''-1} \sum_{j=0}^{m-1} \triangle_j \|e_j\|_{c''}. \label {analmol:S2}
\end{align}
Hence
\begin{align}
\|g_{m+1}\|_{a+\varepsilon}&\lesssim \theta_m^{a+\varepsilon-c-1} \|f\|_c + \theta_m^{a+\varepsilon-c''-1} \sum_{j=0}^{m-1} \triangle_j \|e_j\|_{c''}+ \theta_m^{(a+\varepsilon-c')_{+}} \|e_m\|_{c'}\ , \label{gma}\\
\theta_{m+1}^{(2+a-\alpha)_{+}} \|g_{m+1}\|_{\varepsilon}&\lesssim \theta_m^{\varepsilon-C-1+(2+a-\alpha)_{+}} \|f\|_C + \theta_m^{\varepsilon-C''-1+(2+a-\alpha)_{+}} \sum_{j=0}^{m-1} \triangle_j \|e_j\|_{C''} \nonumber\\
& \qquad \quad + \theta_m^{(\varepsilon-C')_{+}+(2+a-\alpha)_{+}} \|e_m\|_{C'}\ , \label{gmeps}
\end{align}
giving a bound on $\dot u_{m+1}$.  If we choose $c=\alpha+\varepsilon$, $C=\alpha-a+\varepsilon+(2+a-\alpha)_+\leq \alpha+\varepsilon$, the terms involving $f$
are dominated by $\theta_{m+1}^{E} \|f\|_{\alpha + \epsilon}$\\

To estimate $e_j$, we will consider the smoothing error $e'_j$ and the linearization error $e_j''$ separately. At the end of this section we use (\ref{Phi215}) and (\ref{psi216}) to bound $e'_j$ resp.~$e_j''$ in terms of the H\"{o}lder norms
of $u_j$, $v_j$ and $u_j-v_j$ for $0\leq j \leq m$, which are controlled by the induction hypothesis. A computation eventually results in
$$\|\dot{u}_{m+1}\|_{a+\epsilon} \leq C \theta_{m+1}^{E} \left(A \|f\|_{\alpha + \epsilon} + \delta \sigma(1+\delta+\delta^2)\right)$$
for small $\sigma \in (0,1)$ whenever $\theta_0 =\theta_0(\sigma, u_0, \alpha, S_\theta)$ is sufficiently large. $A$ only depends on the constants in (\ref{psi216}) and Property (iv) of the smoothing operator. Both $\theta_0$ and $A$ are, in principle, explicit. We choose $\sigma = \frac{1}{2C}\frac{1}{1+\delta+\delta^2}$.\\

Then for all $f$ in the ball $\lbrace u: \|u\|_{\alpha + \epsilon} \leq \frac{\delta}{2AC} \rbrace$ we have
\begin{equation}\label{analmol:dotuk1first}
 \|\dot u_{m+1}\|_{a+\epsilon} \leq \delta \theta_{m+1}^E.
\end{equation}

On the other hand, in the first step the solution to the linearized problem $\dot u_0=\triangle_0^{-1}\Psi(S_{\theta_0}u_0) S_{\theta_0} f$ is easily estimated using (\ref{psi216}) and the smoothing properties
\begin{align*}
 \triangle_0 \|\dot u_0\|_{a+\epsilon}&\leq C' (\|S_{\theta_0}f\|_{a+\epsilon} + \|S_{\theta_0}f\|_{\epsilon} \|S_{\theta_0}u_0\|_{a+2+\epsilon})\\
&\leq C'' \theta_0(1+\theta_0^{-a} \|u_0\|_{2+\epsilon+a}) \|f\|_{\alpha + \epsilon} \theta_0^{E}.
\end{align*}
We now denote by $\mathfrak{C}$ the maximum of $C''\triangle_0^{-1} \theta_0(1+\theta_0^{-a} \|u_0\|_{2+\epsilon+a})$ and the previous constant $2AC$ and choose $\delta= \mathfrak{C}\|f\|_{\alpha+\epsilon}$. Since $\|f\|_{\alpha+\epsilon}$ was sufficiently small by hypothesis, so is $\delta$, and (\ref{analmol:dotuk1first}) is fulfilled. By induction, we deduce
\begin{equation}\label{analmol:dotuk1second}
 \|\dot u_{m+1}\|_{a+\epsilon} \leq \mathfrak {C} \|f\|_{\alpha+\epsilon} \theta_{m+1}^E \quad \forall\:k\geq 0.
\end{equation}
If $u$ denotes the exact solution, we obtain
\begin{align*}
\|u-u_m\|_{a+\epsilon} &\leq \sum_{j=m+1}^\infty \triangle_j \|\dot u_j \|_{a+\epsilon} \leq \mathfrak {C} \|f\|_{\alpha+\epsilon} \sum_{j=m+1}^\infty \triangle_j \theta_j^{E}\\
&\leq C_{\tau} \mathfrak {C} \|f\|_{\alpha+\epsilon} \theta_m^{E+1+\tau}
\end{align*}
for any $\tau >0$ small such that $E+1+\tau <0$. As $u$ was the exact solution, this yields the assertion of Theorem \ref{apriorithm}.\\

Note that
$$\Phi(u_{m+1}) - \Phi(u) =  \Phi(u_{m+1}) - \Phi(u_0) - f =
(S_{\theta_m} f -f) + \triangle_m e_m + (E_m-S_{\theta_m} E_m)\ .$$
We have shown that the right hand side converges to $0$ in $\mathcal{H}^{a+\epsilon}$, so that also
$\Phi(u_{m})$ converges to $\Phi(u)$.\\

To complete the proof, it remains to estimate $e_m$ and to translate the result into a bound on $\dot u_{m+1}$. For the $j$--th smoothing error
$$e'_j = (\Phi'(u_j)-\Phi'(v_j))\dot u_j = \int_0^1 \Phi''(v_j + t(u_j-v_j); u_j-v_j, \dot u_j)\ dt\ ,$$
equation \eqref{Phi215} implies
\begin{align*}
\|e_j'\|_{2\epsilon+a} &\lesssim \|u_j-v_j\|_{2+3\epsilon+a} \|\dot u_j\|_0 + \|u_j-v_j\|_{0}\|\dot u_j\|_{2+3\epsilon+a}\\
&\qquad\,\, +2\|u_j-v_j\|_{0}\|\dot u_j\|_0(\|u_j\|_{3+2\epsilon+a}+\|v_j\|_{3+2\epsilon+a}) \ .
\end{align*}
Using Lemma \ref{analmol:lemma1}, $\|u_j\|_b\leq \|u_0\|_b + C \delta$, and the estimates for $v_j$ and $u_j-v_j$ from \eqref{analmol:star}, \eqref{analmol:twostar}, we obtain
\begin{align}\label{ep}
\|e_j'\|_{2\epsilon+a} &\lesssim (\lVert u_0\rVert_b \theta_j^{2+3\epsilon+a-b}+\delta \theta_j^{2+2\epsilon+a-\alpha}) \delta \theta_j^{-\alpha-1} + (\|u_0\|_{c_1}\theta_j^{-c_1}+\delta \theta_j^{-\alpha-\epsilon}) \delta \theta_j^{1+2\epsilon+a-\alpha}  \\
& \qquad+ (\|u_0\|_{c_2}\theta_j^{-c_2}+\delta \theta_j^{-\alpha-\epsilon})\delta \theta_j^{-\alpha-1} (\|u_0\|_{3+2\epsilon+a} +\delta +\delta \theta_j^{(a-\alpha-\epsilon)_{+}})  \ .\nonumber
\end{align}
Similarly, as the remainder of the first Taylor approximation, $e''_j$ is also controlled by $\Phi''$, and analogous estimates result in
\begin{align}\label{epp}
\|e''_j\|_{2\epsilon+a} &\lesssim \triangle_j \lbrace \|\dot u_j\|_{2+3\epsilon+a}\|\dot u_j\|_0 +\|\dot u_j\|^2_0(\|u_j\|_{3+2\epsilon+a}+\|\dot u_j\|_{3+2\epsilon+a})\rbrace \\
&\lesssim \triangle_j \lbrace \delta \theta_j^{1+3\epsilon+a-\alpha} \delta \theta_j^{-\alpha-1}+(\delta\theta_j^{-\alpha-1})^2 ( \|u_0\|_{3+2\epsilon+a} +\delta +\delta \theta_j^{(3+\epsilon+a-\alpha)_+}\rbrace . \nonumber
\end{align}
To estimate $g_{m+1}$ in \eqref{gma} and \eqref{gmeps}, it remains bound sums $\sum_{j=0}^{m-1} \triangle_j \left\{\|e'_j\|_{b}+\|e''_j\|_{b}\right\}$. We consider a generic term of the form $\sum_{j=0}^{m-1} \triangle_j
\theta_j^{-d} F(\delta,u_0)$ for suitable $F$ obtained from \eqref{ep} resp.~$\sum_{j=0}^{m-1} \triangle_j^2 \theta_j^{-d} F(\delta,u_0)$ from \eqref{epp}. Concerning the former,
if $d>1$, we have for any small $\tau>0$
\begin{align*}
\sum_{j=0}^{m-1} \triangle_j \theta_j^{-d} &\leq \theta_0^{-d+1+\tau} \sum_{j=0}^{m-1} \triangle_j \theta_j^{-1-\tau} \leq C_\tau \theta_0^{-d+1+\tau},
\end{align*}
with $C_\tau$ independent of $\theta_0\geq \theta_0^{\min} >0$ and $\kappa> \kappa_{\min}>0$. Here we have used that
\begin{equation*}\label{neededtauabsch}
\triangle_j\theta_j^{-1-\tau} \lesssim \kappa^{-1} \theta_j^{1-\kappa-1-\tau}=\kappa^{-1} (\theta_0^{\kappa}+j)^{-1-\frac{\tau}{\kappa}}\ .
\end{equation*}
For $d<1$,
\begin{align*}
\sum_{j=0}^{m-1} \triangle_j \theta_j^{-d} &\leq \theta_k^{-d+1+\tau} \sum_{j=0}^{m-1} \triangle_j \theta_j^{-1-\tau} \leq  C_\tau \theta_m^{-d+1+\tau}\ .
\end{align*}
Finally, for $d=1$
\begin{align*}
\sum_{j=0}^{m-1} \triangle_j \theta_j^{-1} \leq C_\tau \theta_m^{\tau}.
\end{align*}
As for the term coming from \eqref{epp}
\begin{align*}
 \sum_{j=0}^{m-1} \triangle_j^2 \theta_j^{-d} \leq C_\tau \kappa^{-1} \theta_0^{2- d-\kappa+\tau}\ .
\end{align*}
Estimating the sums in \eqref{gma} resp.~\eqref{gmeps} thus increases the exponent on $\theta_m$ resp.~$\theta_0$ in the estimates of $e'_m$ by at most $1+\tau$.
From the estimates of $e''_m$, one obtains $\theta_0$ raised to a power which is arbitrarily negative for large $\kappa$.\\
As a result
\begin{align} \label{epestimate}
\theta_m^{a+\varepsilon-c''-1} \sum_{j=0}^{m-1} \triangle_j \|e'_j\|_{c''} &\lesssim \delta \|u_0\|_b \theta_m^{a+\varepsilon-c''-1} \theta_{m/0}^{3+\varepsilon+c''-b+\tau}+\delta^2 \theta_m^{a+\varepsilon-c''-1} \theta_{m/0}^{2+c''-2\alpha+\tau}\nonumber \\
& \quad +\delta \|u_0\|_{c_1} \theta_m^{a+\varepsilon-c''-1}\theta_{m/0}^{2-c_1+c''-\alpha+\tau} + \delta^2\theta_m^{a+\varepsilon-c''-1}\theta_{m/0}^{2-\varepsilon+c''-2\alpha+\tau}
\nonumber \\
&\quad+ \delta \|u_0\|_{c_2}\|u_0\|_{3+c''} \theta_m^{a+\varepsilon-c''-1}\theta_{m/0}^{-c_2-\alpha+\tau} \nonumber\\
& \quad + \delta^2\|u_0\|_{3+c''}\theta_m^{a+\varepsilon-c''-1}\theta_{m/0}^{-\varepsilon-2\alpha+\tau}+ \delta^2 \|u_0\|_{c_2} \theta_m^{a+\varepsilon-c''-1}\theta_{m/0}^{-c_2-\alpha+\tau}\nonumber \\
& \quad + \delta^3\theta_m^{a+\varepsilon-c''-1}\theta_{m/0}^{-\varepsilon-2\alpha+\tau}+ \delta^2 \|u_0\|_{c_2} \theta_m^{a+\varepsilon-c''-1}\theta_{m/0}^{-c_2-\alpha+(c''-\alpha-3\varepsilon)_++\tau}\nonumber \\
& \quad + \delta^3\theta_m^{a+\varepsilon-c''-1}\theta_{m/0}^{-\varepsilon-2\alpha+(c''-\alpha-3\varepsilon)_++\tau} \ .
\end{align}
Here $\theta_{m/0}$ is $\theta_{m}$ or $\theta_{0}$, depending on whether its exponent is greater or smaller $\tau$. Choosing e.g.~$c''=\alpha+2\varepsilon$, $b=3+\varepsilon+c''+2\tau$, $c_1=2+c''-\alpha+2\tau$ and $c_2=0$, the exponents of $\theta_{m/0}$ are negative and the exponent of
$\theta_m$ is strictly smaller than $E=a-\alpha-1$.
Similarly, we obtain
\begin{align}
\theta_m^{a+\varepsilon-c''-1} \sum_{j=0}^{m-1} \triangle_j \|e''_j\|_{c''} &\lesssim \delta^2 \theta_m^{a+\varepsilon-c''-1}  \theta_{0}^{\varepsilon + c''-2\alpha+2-\kappa+\tau}+\delta^2 (\|u_0\|_{3+c''} + \delta) \theta_m^{a+\varepsilon-c''-1} \theta_{0}^{-2\alpha-\kappa+\tau}
\nonumber \\ & \quad +\delta^3 \theta_m^{a+\varepsilon-c''-1} \theta_{0}^{-2\alpha-\kappa+\tau}\ ,
\end{align}
where the exponents of the $\theta_0$ and $\theta_m$ have the same properties as in \eqref{epestimate}. \\
It remains to estimate the term $\theta_m^{(a+\varepsilon-c')_{+}} \|e_m\|_{c'}$ in \eqref{gma}. We choose $c'=a+\varepsilon$ and, in \eqref{ep}, set $c_1=c_2$ equal to the above $c''$ obtain
\begin{align*}
\theta_m^{(a+\varepsilon-c')_{+}} \|e_m\|_{c'}& \lesssim(\lVert u_0\rVert_b \theta_m^{2+2\epsilon+a-b}+\delta \theta_m^{2+\epsilon+a-\alpha}) \delta \theta_m^{-\alpha-1} +
(\|u_0\|_{c''}\theta_m^{-c''}+\delta \theta_m^{-\alpha-\epsilon}) \delta \theta_m^{1+\epsilon+a-\alpha}  \\
& \qquad+ (\|u_0\|_{c''}\theta_m^{-c''}+\delta \theta_m^{-\alpha-\epsilon})\delta \theta_m^{-\alpha-1} (\|u_0\|_{3+\epsilon+a} +\delta +\delta \theta_m^{(a-\alpha)_{+}}) \\
& \qquad + \delta^2 \triangle_m \theta_m^{2\epsilon+a-2\alpha} + \delta^2 \triangle_m \theta_m^{-2\alpha-2}(\|u_0\|_{3+\varepsilon+a} +\delta)+\delta^3 \triangle_m \theta_m^{-2\alpha-2+(3+a-\alpha)_+}.
\end{align*}
The exponent of $\theta_m$ is again strictly smaller than $E$.\\
The analysis of \eqref{gmeps} is analogous. This completes the proof of Theorem \ref{apriorithm}.
\end{proof}

\subsection{Proof of Theorem \ref{heatthem}}

We consider the operator $A$ as an unbounded operator on the H\"older space $\mathscr {H}^a$ with domain $D(A)=\mathscr {H}^{a+2}$ (if $a \notin \mathbb N_0$).
Using the nonpositivity of $A$ and \cite[Theorem 9.3]{Shubin}, we see that $A-\lambda$ is invertible for $\lambda \in \mathcal S_{\theta}= \lbrace \lambda \in \mathbb C\setminus\{0\}: |\mbox{arg} \lambda| <\theta  \rbrace$, $\theta \in (\pi/2,\pi)$, and
that $(A-\lambda)^{-1}$ is a pseudodifferential operator, depending on the parameter $\lambda$, whose symbol decays as $\frac{C}{|\lambda|}$. The mapping properties  \cite[Proposition 8.6]{Taylor3} of such
operators in H\"older spaces, which are analogous to those for Sobolev spaces, therefore imply
\begin{equation}\label{ineqonhoelder}
\|(A-\lambda)^{-1} u\|_{a} \leq \frac {C}{|\lambda|} \|u\|_{a}, \quad \forall\: \lambda \in S_{\theta,0}.
\end{equation}
Equation (\ref{ineqonhoelder}) allows to define the analytic semigroup generated by $A$,
\begin{equation}
e^{tA}u:= \frac{1}{2\pi i} \int_{\gamma_{r,\eta}} e^{t\lambda} (A -\lambda)^{-1} u \,d\lambda, \quad t>0,
\end{equation}
where $r>0, \eta \in ]\pi/2,\pi[ $, and $\gamma_{r,\eta}$ is the curve $\lbrace \lambda \in \mathbb C: |\mbox{arg} \lambda|=\eta, |\lambda| \geq r\rbrace \cup
\lbrace \lambda \in \mathbb C: |\mbox{arg} \lambda|\leq \eta, |\lambda| = r\rbrace $, oriented counterclockwise. $e^{tA}u$ does not depend on the choice of $r$ and $\eta$.
We recall some basic properties of analytic semigroups (Proposition 2.1.1, \cite{Lunardi}):
\begin{proposition}\label{propolunardi}
\begin{itemize}
     \item[(i)] $\|e^{tA} u\|_a \leq C_0 \|u\|_a, \quad \forall t \geq 0$.
     \item[(ii)] $e^{tA}e^{sA}=e^{(t+s)A},\quad \forall\:t, s \geq 0$.
     \item[(iii)] $\lim \limits_{t \rightarrow 0^{+}} \|e^{tA}u - u\|_a=0,\quad \forall\: u \in \overline{D(A)}$.
     \item[(iv)] There are constants $C_l$, such that
                      \begin{equation}
                     \|t^l A^l e^{tA} u\|_{a} \leq C_l \|u\|_a, \quad 0<t\leq 1.
                      \end{equation}
    \item[(v)] $t \mapsto e^{tA}$ is a real-analytic function from $(0,\infty)$ to the Banach space of bounded linear operators on $\mathcal{H}^a$ (with norm given by the operator norm) and
                       \begin{equation}\label{smooth:abl}
                        \frac{d^l}{dt^l}e^{tA}=A^l e^{tA},\quad t>0.
                       \end{equation}
\end{itemize}
\end{proposition}
Concerning Theorem \ref{heatthem}, we first consider Property (0).
Using Proposition \ref{propolunardi}(iii) and setting $S_\theta=e^{tA}$ and $t=\theta^{-2k}$ we have
\begin{equation}
\lim \limits_{\theta \rightarrow \infty} S_\theta u=u, \quad \forall u \in \overline{\mathscr{H}^{2+a}}
\end{equation}
and thus, Property (0) holds.\\
Using Proposition \ref{propolunardi}(i) and the fact that $S_\theta=e^{tA}$ is a continuous operator on $\mathscr {H}^{b}$ we have
\begin{equation*}
\|e^{tA} u\|_b \lesssim \|u\|_b \lesssim \|u\|_a, \quad  \forall \,b \leq a
\end{equation*}
and thus also Property (i).\\
In order to prove Property (ii), note that it suffices to show the assertion for $0 < t \leq 1$, or equivalently $\theta\geq1$.
We use that $(A -1)^{-1} : \mathscr {H}^a \rightarrow \mathscr {H}^{a+2k}$ is continuous, $\|(A-1)^{-1} u\|_{\mathscr {H}^{a+2k}}\lesssim \|u\|_{\mathscr {H}^a}$. We then have
\begin{equation*}
\|v\|_{a+2k} \lesssim \|(A-1)v\|_a \lesssim \|Av\|_a +  \|v\|_a.
\end{equation*}
We first set $l=1$ and $v=t e^{tA}u$ and deduce
\begin{align}\label{smooth:eqnhilfs}
\|t e^{tA}u\|_{a+2k} \lesssim \|(A - 1)t e^{tA}u\|_a \lesssim \|t A e^{tA}u\|_a + \|te^{tA}u\|_a
\end{align}
and by using  Proposition \ref{propolunardi}(i) and  Proposition \ref{propolunardi}(iv)   we have
$$
\|te^{tA}u\|_a \leq \|e^{tA}u\|_a \lesssim  \|u\|_a, \quad 0 < t \leq 1,
$$
and finally by (\ref{smooth:eqnhilfs}) we obtain
\begin{equation*}
\|t e^{tA}u\|_{a+2k} \lesssim \|u\|_a.
\end{equation*}
By iterating this argument $l$-times using
\begin{equation*}
\|t^l e^{tA}u\|_{a+2kl} =l^{l} \|\frac{t}{l} e^{t/l}A \cdot_{ \dots} \cdot  \frac{t}{l} e^{t/l}A u\|_{a+2kl}
\end{equation*}
we have
\begin{equation*}
\|t^l e^{tA}u\|_{a+2kl} \lesssim \|u\|_a \ .
\end{equation*}
Setting $b= a+2kl$, $t= \theta^{-2k}$, Property (ii) holds for this specific $b$.\\
For an arbitrary $b$,  $\tilde b:= a+2kl \geq b$, write $b=\sigma a +(1-\sigma)\tilde b$. The interpolation estimate (\ref{interpol}) gives
\begin{equation*}
\|e^{tA} u\|_b \leq \|e^{tA} u\|_a^{\lambda} \|e^{tA} u\|_{\tilde b}^{1-\lambda} \lesssim  t^{-l(1-\lambda)} \|u\|_a^\lambda\ \|u\|_a^{1-\lambda}\ ,
\end{equation*}
and we deduce
\begin{equation*}
\|e^{tA} u\|_b \lesssim t^{-(1-\lambda)l} \|u\|_a= t^{-(b-a)/2k} \|u\|_a.
\end{equation*}
Setting now $S_\theta:=e^{tA}$ with $ t =\theta^{-2k}$ we have proved
\begin{equation*}
\|S_\theta u\|_b \lesssim \theta^{b-a} \|u\|_a
\end{equation*}
and thus, Property (ii) holds.\\
For Property (iv) we first use $ t =\theta^{-2k}$ and observe
\begin{equation*}
\frac{d} {d\theta} e^{tA} u= \frac{dt}{d\theta} \frac{d}{dt} e^{tA}u=-2k t^{1/2k} (tAe^{tA}u)=-\frac{2k}{\theta} tAe^{tA}u.
\end{equation*}
The same proof as for Property (ii) yields, setting $S_\theta=e^{tA}$:
\begin{equation*}
\big \| \frac{d}{d\theta}S_\theta u\big \|_b  = \frac{2k}{\theta} \|tAe^{tA} u\|_b \lesssim  \frac{2k}{\theta} \theta^{b-a} \|u\|_a = 2k \theta^{b-a-1} \|u\|_a\ .
\end{equation*}

Finally, given the continuity of $S_\theta$ on $\mathscr {H}^a$, it suffices to show Property $(iii')$ for $b \neq a$. Note that $1-A: \mathscr {H}^{a+2k} \rightarrow \mathscr {H}^{a}$ is an isomorphism.
With $1-e^{t\lambda} = -\lambda \int_0^t e^{\lambda s}\,ds$, we have for $u \in C^\infty$, $v=(1-A)^{\frac{a-b}{2k}}u$ and $t\in (0,1]$
\begin{eqnarray*}
u- e^{tA} u &=& \frac{1}{2\pi i} \int_{\gamma_{r,\eta}} (1-e^{t\lambda}) (A -\lambda)^{-1} (1-A)^{-\frac{a-b}{2k}} v \,d\lambda\\\
&=& -\frac{1}{2\pi i} \int_{\gamma_{r,\eta}} \lambda \int_0^t e^{\lambda s}\,ds\ (A -\lambda)^{-1} (1-\lambda)^{-\frac{a-b}{2k}} v \,d\lambda \ .
\end{eqnarray*}
The double integral is absolutely convergent for $\frac{a-b}{2k} \in (0,1)$. After interchanging the order of integration and using the triangle inequality as well as $\|(A -\lambda)^{-1}v\|_b \lesssim \frac{\|v\|_b}{\lambda}$, the right hand side is smaller than a constant times
\begin{equation*}
\int_0^t \int_{\gamma_{r,\eta}} e^{s\ \mathrm{Re}\ \lambda} |1-\lambda|^{-\frac{a-b}{2k}} \|v\|_b \,|d\lambda| \,ds\ .
\end{equation*}
If $r<1$, we may bound $|1-\lambda|^{-1} \leq C_r (1+|\lambda|)^{-1}$ for all $\lambda \in \gamma_{r,\eta}$. It therefore remains to estimate
\begin{equation*}
\int_0^t \int_{\gamma_{r,\eta}} e^{s\ \mathrm{Re}\ \lambda} (1+|\lambda|)^{-\frac{a-b}{2k}} \|v\|_b \,|d\lambda| \,ds\ .
\end{equation*}
We split the integral $\int_{\gamma_{r,\eta}} = I_r + I_+ + I_-$ into integrals over $\gamma_r = \lbrace \lambda = r e^{i \sigma}\in \mathbb C: |\sigma|\leq \eta\rbrace$, $\gamma_+ = \lbrace \lambda=\rho e^{i \eta} \in \mathbb C: \rho \geq r\rbrace$ resp.~$\gamma_-=\lbrace \lambda=\rho e^{-i \eta} \in \mathbb C: \rho \geq r\rbrace$ and consider the three terms separately.
The first integral,
\begin{equation*}
I_r=\int_0^t \int_{-\eta}^{\eta} e^{s\cos(\sigma)}  \,d\sigma \,ds\ (1+r)^{-\frac{a-b}{2k}} \|v\|_b\ ,
\end{equation*}
is bounded by $t (2 \eta) e  (1+r)^{-\frac{a-b}{2k}} \|v\|_b$ and hence of order $t$.
For the second and third integrals,
\begin{equation*}
I_{\pm}=\int_0^t \int_{r}^{\infty} e^{-s\rho |\cos(\eta)|} (1+\rho)^{-\frac{a-b}{2k}} \|v\|_b\  \,d\rho \,ds
\end{equation*}
the change of variables $\rho \mapsto \frac{\rho}{s |\cos(\eta)|}$ leads to
\begin{equation*}
I_{\pm}=\int_0^t \int_{sr|\cos(\eta)|}^{\infty} e^{-\rho} \left(\frac{s |\cos(\eta)|}{s |\cos(\eta)|+\rho}\right)^{\frac{a-b}{2k}} \|v\|_b\  \frac{1}{s |\cos(\eta)|} \,d\rho \,ds \ ,
\end{equation*}
or
\begin{equation*}
I_{\pm}\leq \int_0^t (s |\cos(\eta)|)^{\frac{a-b}{2k}-1}\int_{0}^{\infty} e^{-\rho} {\rho}^{-\frac{a-b}{2k}} \|v\|_b\, \,d\rho \,ds \ \lesssim t^{\frac{a-b}{2k}} \|v\|_b\ .
\end{equation*}
Using $t=\frac{1}{\theta^{2k}}$ and $\|v\|_b=\|(1-A)^{\frac{a-b}{2k}}u\|_b\lesssim\|u\|_a$, (iii') follows.

\section{Acknowledgments}
We thank Lothar Banz for instructive discussions about computational aspects of this work. This work was supported by the cluster of excellence QUEST, the Danish National Research Foundation (DNRF) through the Centre for Symmetry and Deformation and the Danish Science Foundation (FNU) through research grant 10-082866. H.~G.~thanks the Institut f\"{u}r Angewandte Mathematik for hospitality.
\bibliographystyle{unsrt}
 \bibliography{molo_final.bib}

\end{document}